\documentclass[psamsfonts,11pt]{amsart}
\usepackage[english]{babel}
\usepackage{amsmath,amsthm,amsfonts,amssymb,epsfig}
\usepackage{fancyhdr, lastpage}
\usepackage{mathtools}
\usepackage{graphicx}
\usepackage{enumerate}
\usepackage[left=1in,top=1in,right=1in]{geometry}
\usepackage{tikz}
\usepackage{verbatim}
\usepackage{mdwlist}
\usepackage{parskip}
\usepackage{hyperref}
\usepackage{mdframed}
\usepackage{subfigure}
\usepackage{float}
\usepackage{caption}
\usepackage{listings}
\usepackage{color}
\usepackage{blkarray}

\def\Big#1{\makebox(0,0){\huge#1}}

\newtheorem{thm}{Theorem}[subsection]
\newtheorem{lem}[thm]{Lemma}
\newtheorem{prop}[thm]{Proposition}
\newtheorem{cor}[thm]{Corollary}
\theoremstyle{remark}
\newtheorem*{df}{Definition}

\newcommand{\eat}[1]{}
\newtheorem{property}[thm]{Property}

\DeclarePairedDelimiter{\seq}{\{}{\}}

\numberwithin{equation}{section}
\makeatletter
\def\l@section{\@tocline{1}{0pt}{1pc}{}{}}
\def\l@subsection{\@tocline{2}{0pt}{1pc}{4.6em}{}}
\def\l@subsubsection{\@tocline{3}{0pt}{1pc}{7.6em}{}}
\renewcommand{\tocsection}[3]{%
  \indentlabel{\@ifnotempty{#2}{%
    \ignorespaces\@ifnotempty{#1}{\:}\textbf{#1} \makebox[1.5em][l]{\textbf{#2}.\hfill}}}\textbf{#3}}
\renewcommand{\tocsubsection}[3]{%
  \indentlabel{\@ifnotempty{#2}{\hspace*{2.3em}\makebox[2.3em][l]{%
    \ignorespaces#1 #2.\hfill}}}#3}
\renewcommand{\tocsubsubsection}[3]{%
  \indentlabel{\@ifnotempty{#2}{\hspace*{4.6em}\makebox[3em][l]{%
    \ignorespaces#1 #2.\hfill}}}#3}
\makeatother

\let\oldtocsection=\tocsection
\let\oldtocsubsection=\tocsubsection
\renewcommand{\tocsection}[2]{\hspace{0em}\oldtocsection{#1}{#2}}
\renewcommand{\tocsubsection}[2]{\hspace{-1em}\oldtocsubsection{#1}{#2}}
\begin{document}
\title[Recognizing signed-graphic matroids]{Recognizing signed-graphic matroids: Cylinder flips and the importance of column scaling}
\author[Lisa Seung-Yeon Lee]{Lisa Seung-Yeon Lee\\\\May 6, 2014}
\email{lilee@princeton.edu}
\maketitle
\vspace{-.33in}
\begin{abstract}
In this paper, we investigate the importance of column scaling in relating two signed-graphic representations of the same matroid. We used the Sage Mathematics software to generate many examples of signed-graphic matroids and their signed-graphic representations. Our examples show that column scaling is sometimes necessary in order to transform one signed-graphic representation into another; moreover, there exist many collections of signed-graphic representations that row-reduce to the same standard form. We also discuss an interesting matroid-preserving operation on a signed graph, which we call the cylinder flip, that relates certain pairs of signed-graphic representations of the same matroid.
\end{abstract}

\vspace{.05in}
\hskip0.2in{\sc Contents}
\vskip.07in
\renewcommand\contentsname{}
\tableofcontents

\newsavebox{\diffSomeone}
\savebox{\diffSomeone}{\scriptsize $\left(\begin{array}{rrrrrrrrrr}
0 & -1 & 0 & 0 & 0 & -1 & 0 & 0 & 0 & -1 \\
0 & 0 & 0 & -1 & 0 & 0 & -1 & 0 & 0 & 1 \\
1 & 0 & -1 & 1 & 0 & -1 & 0 & 0 & 0 & 0 \\
0 & -1 & 1 & 0 & 0 & 0 & 0 & -1 & 0 & 0 \\
1 & 0 & 0 & 0 & 1 & 0 & 0 & 1 & 1 & 0 \\
0 & 0 & 0 & 0 & 1 & 0 & 1 & 0 & -1 & 0
\end{array}\right)$}

\newsavebox{\diffSometwo}
\savebox{\diffSometwo}{\scriptsize $\left(\begin{array}{rrrrrrrrrr}
0 & -1 & 0 & -1 & 0 & -1 & -1 & 0 & 0 & 0 \\
1 & 0 & -1 & 1 & 0 & -1 & 0 & 0 & 0 & 0 \\
0 & -1 & 1 & 0 & 0 & 0 & 0 & -1 & 0 & 0 \\
1 & 0 & 0 & 0 & 1 & 0 & 0 & 0 & 0 & 1 \\
0 & 0 & 0 & 0 & 0 & 0 & 0 & -1 & -1 & 1 \\
0 & 0 & 0 & 0 & 1 & 0 & 1 & 0 & -1 & 0
\end{array}\right)$}

\newsavebox{\diffSomethree}
\savebox{\diffSomethree}{\scriptsize $\left(\begin{array}{rrrrrrrrrr}
0 & 0 & 0 & -1 & 1 & 0 & 0 & 1 & 0 & 0 \\
1 & 0 & 0 & 1 & 0 & 0 & 0 & 0 & 1 & 0 \\
0 & 0 & 0 & 0 & 1 & 0 & 1 & 0 & -1 & 0 \\
0 & -1 & 0 & 0 & 0 & -1 & 0 & 0 & 0 & -1 \\
1 & 0 & -1 & 0 & 0 & -1 & -1 & 0 & 0 & 1 \\
0 & -1 & 1 & 0 & 0 & 0 & 0 & -1 & 0 & 0
\end{array}\right)$}

\newsavebox{\diffSomefour}
\savebox{\diffSomefour}{\scriptsize $\left(\begin{array}{rrrrrrrrrr}
0 & 0 & 1 & 0 & 0 & 1 & 0 & 0 & 1 & 0 \\
0 & -1 & 1 & 0 & 0 & 0 & 0 & -1 & 0 & 0 \\
-1 & 1 & 0 & 0 & -1 & 1 & 0 & 0 & 0 & 0 \\
0 & 0 & 0 & 0 & 1 & 0 & 1 & 0 & -1 & 0 \\
-1 & 0 & 0 & -1 & 0 & 0 & 0 & 1 & 0 & -1 \\
0 & 0 & 0 & -1 & 0 & 0 & -1 & 0 & 0 & 1
\end{array}\right)$}

\newsavebox{\diffSomefive}
\savebox{\diffSomefive}{\scriptsize $\left(\begin{array}{rrrrrrrrrr}
0 & 0 & 1 & 0 & 0 & 1 & 0 & 0 & 1 & 0 \\
0 & -1 & 1 & 0 & 0 & 0 & 0 & -1 & 0 & 0 \\
-1 & 1 & 0 & 0 & 0 & 1 & 1 & 0 & -1 & 0 \\
-1 & 0 & 0 & 0 & -1 & 0 & 0 & 0 & 0 & -1 \\
0 & 0 & 0 & -1 & 0 & 0 & -1 & 0 & 0 & 1 \\
0 & 0 & 0 & -1 & 1 & 0 & 0 & 1 & 0 & 0
\end{array}\right)$}

\newsavebox{\diffSomesix}
\savebox{\diffSomesix}{\scriptsize $\left(\begin{array}{rrrrrrrrrr}
0 & 0 & 1 & 0 & 0 & 1 & 0 & 0 & 1 & 0 \\
0 & -1 & 1 & 0 & 0 & 0 & 0 & -1 & 0 & 0 \\
0 & 1 & 0 & 0 & 0 & 1 & 0 & 0 & 0 & 1 \\
1 & 0 & 0 & 0 & 0 & 0 & -1 & 0 & 1 & 1 \\
-1 & 0 & 0 & -1 & -1 & 0 & -1 & 0 & 0 & 0 \\
0 & 0 & 0 & 1 & -1 & 0 & 0 & -1 & 0 & 0
\end{array}\right)$}

\newsavebox{\diffSomeseven}
\savebox{\diffSomeseven}{\scriptsize $\left(\begin{array}{rrrrrrrrrr}
0 & 0 & 0 & -1 & 0 & 0 & -1 & 1 & 1 & 0 \\
1 & 0 & 0 & 1 & 0 & 0 & 0 & 0 & 1 & 0 \\
0 & 0 & -1 & 0 & -1 & -1 & -1 & 0 & 0 & 0 \\
1 & 0 & 0 & 0 & 1 & 0 & 0 & 0 & 0 & 1 \\
0 & 1 & 0 & 0 & 0 & 1 & 0 & 0 & 0 & 1 \\
0 & 1 & -1 & 0 & 0 & 0 & 0 & 1 & 0 & 0
\end{array}\right)$}

\newsavebox{\diffSomeeight}
\savebox{\diffSomeeight}{\scriptsize $\left(\begin{array}{rrrrrrrrrr}
0 & 0 & 0 & 0 & 0 & 0 & 0 & 1 & 1 & -1 \\
0 & 0 & 0 & -1 & 0 & 0 & -1 & 0 & 0 & 1 \\
1 & 0 & 0 & 1 & 0 & 0 & 0 & 0 & 1 & 0 \\
0 & 0 & -1 & 0 & -1 & -1 & -1 & 0 & 0 & 0 \\
1 & -1 & 0 & 0 & 1 & -1 & 0 & 0 & 0 & 0 \\
0 & 1 & -1 & 0 & 0 & 0 & 0 & 1 & 0 & 0
\end{array}\right)$}

\newsavebox{\diffSomenine}
\savebox{\diffSomenine}{\scriptsize $\left(\begin{array}{rrrrrrrrrr}
0 & 0 & 1 & 0 & 0 & 1 & 0 & 0 & 1 & 0 \\
0 & -1 & 1 & 0 & 0 & 0 & 0 & -1 & 0 & 0 \\
0 & 1 & 0 & 1 & 0 & 1 & 1 & 0 & 0 & 0 \\
1 & 0 & 0 & 1 & 0 & 0 & 0 & 0 & 1 & 0 \\
-1 & 0 & 0 & 0 & -1 & 0 & 0 & 0 & 0 & -1 \\
0 & 0 & 0 & 0 & -1 & 0 & -1 & -1 & 0 & 1
\end{array}\right)$}

\newsavebox{\diffSomerreftwo}
\savebox{\diffSomerreftwo}{\scriptsize $\left(\begin{array}{rrrrrrrrrr}
1 & 0 & 0 & 0 & 0 & 0 & -1 & 0 & 1 & 1 \\
0 & 1 & 0 & 0 & 0 & 1 & 0 & 0 & 0 & 1 \\
0 & 0 & 1 & 0 & 0 & 1 & 0 & 0 & 1 & 0 \\
0 & 0 & 0 & 1 & 0 & 0 & -2 & 0 & 0 & 2 \\
0 & 0 & 0 & 0 & 1 & 0 & 1 & 0 & -1 & 0 \\
0 & 0 & 0 & 0 & 0 & 0 & 0 & 1 & 1 & -1
\end{array}\right)$}

\newsavebox{\diffSometen}
\savebox{\diffSometen}{\scriptsize $\left(\begin{array}{rrrrrrrrrr}
-1 & 0 & 0 & 1 & 0 & 0 & -1 & 1 & 0 & 0 \\
1 & 0 & 0 & 0 & 1 & 0 & 0 & 0 & 0 & 1 \\
0 & 1 & 0 & 0 & 0 & 1 & 0 & 0 & 0 & 1 \\
0 & 1 & -1 & 0 & 0 & 0 & 0 & 1 & 0 & 0 \\
0 & 0 & 1 & 0 & 0 & 1 & 0 & 0 & 1 & 0 \\
0 & 0 & 0 & 0 & 1 & 0 & 1 & 0 & -1 & 0
\end{array}\right)$}

\newsavebox{\diffSomerrefthree}
\savebox{\diffSomerrefthree}{\scriptsize $\left(\begin{array}{rrrrrrrrrr}
1 & 0 & 0 & 0 & 0 & 0 & -1 & 0 & 1 & 1 \\
0 & 1 & 0 & 0 & 0 & 1 & 0 & 0 & 0 & 1 \\
0 & 0 & 1 & 0 & 0 & 1 & 0 & 0 & 1 & 0 \\
0 & 0 & 0 & 1 & 0 & 0 & 1 & 0 & 0 & -1 \\
0 & 0 & 0 & 0 & 1 & 0 & -2 & 0 & 2 & 0 \\
0 & 0 & 0 & 0 & 0 & 0 & 0 & 1 & 1 & -1
\end{array}\right)$}

\newsavebox{\diffSomeeleven}
\savebox{\diffSomeeleven}{\scriptsize $\left(\begin{array}{rrrrrrrrrr}
-1 & 0 & 0 & 0 & 1 & 0 & -1 & -1 & 0 & 0 \\
1 & 0 & 0 & 1 & 0 & 0 & 0 & 0 & 1 & 0 \\
0 & 0 & 0 & 1 & 0 & 0 & 1 & 0 & 0 & -1 \\
0 & 1 & 0 & 0 & 0 & 1 & 0 & 0 & 0 & 1 \\
0 & 1 & -1 & 0 & 0 & 0 & 0 & 1 & 0 & 0 \\
0 & 0 & 1 & 0 & 0 & 1 & 0 & 0 & 1 & 0
\end{array}\right)$}

\newsavebox{\sameAllone}
\savebox{\sameAllone}{\scriptsize $\left(\begin{array}{rrrrrrrrrr}
0 & 0 & 1 & 0 & 1 & 0 & 0 & 0 & 0 & 1 \\
0 & 1 & 0 & 0 & 0 & 0 & -1 & -1 & 0 & 1 \\
0 & 0 & 1 & -1 & 0 & 0 & 0 & 0 & 1 & 0 \\
0 & 1 & 0 & 1 & 0 & 1 & 1 & 0 & 0 & 0 \\
1 & 0 & 0 & 0 & 1 & 0 & 0 & -1 & 0 & 0 \\
1 & 0 & 0 & 0 & 0 & -1 & 0 & 0 & 1 & 0
\end{array}\right)$}

\newsavebox{\sameAlltwo}
\savebox{\sameAlltwo}{\scriptsize $\left(\begin{array}{rrrrrrrrrr}
0 & 0 & 0 & 0 & 0 & -1 & -1 & 0 & 0 & 1 \\
0 & 1 & 0 & 1 & 0 & 0 & 0 & 0 & 0 & 1 \\
0 & 0 & -1 & 1 & 0 & 0 & 0 & 0 & -1 & 0 \\
0 & -1 & 1 & 0 & 1 & 0 & 1 & 1 & 0 & 0 \\
1 & 0 & 0 & 0 & 1 & 0 & 0 & -1 & 0 & 0 \\
1 & 0 & 0 & 0 & 0 & -1 & 0 & 0 & 1 & 0
\end{array}\right)$}

\newsavebox{\sameAllthree}
\savebox{\sameAllthree}{\scriptsize $\left(\begin{array}{rrrrrrrrrr}
0 & 0 & 0 & 0 & 0 & 1 & 1 & 0 & 0 & -1 \\
-1 & 0 & 1 & 0 & 0 & 0 & 0 & 1 & 0 & 1 \\
0 & 0 & 1 & -1 & 0 & 0 & 0 & 0 & 1 & 0 \\
1 & 0 & 0 & 1 & 1 & 0 & 1 & 0 & 0 & 0 \\
0 & -1 & 0 & 0 & 1 & 0 & 0 & 0 & -1 & 0 \\
0 & 1 & 0 & 0 & 0 & 1 & 0 & -1 & 0 & 0
\end{array}\right)$}

\newsavebox{\sameAllfour}
\savebox{\sameAllfour}{\scriptsize $\left(\begin{array}{rrrrrrrrrr}
0 & 0 & 0 & 1 & 0 & 0 & 1 & 1 & 0 & 0 \\
0 & 1 & 0 & 1 & 0 & 0 & 0 & 0 & 0 & 1 \\
0 & 0 & 1 & 0 & 1 & 0 & 0 & 0 & 0 & 1 \\
0 & 1 & 1 & 0 & 0 & 1 & 1 & 0 & 1 & 0 \\
1 & 0 & 0 & 0 & 1 & 0 & 0 & -1 & 0 & 0 \\
1 & 0 & 0 & 0 & 0 & -1 & 0 & 0 & 1 & 0
\end{array}\right)$}

\newsavebox{\sameAllfive}
\savebox{\sameAllfive}{\scriptsize $\left(\begin{array}{rrrrrrrrrr}
0 & 0 & 0 & 1 & 0 & 0 & 1 & 1 & 0 & 0 \\
1 & 0 & -1 & 1 & 0 & -1 & 0 & 0 & 0 & 0 \\
0 & 0 & 1 & 0 & 1 & 0 & 0 & 0 & 0 & 1 \\
-1 & 0 & 0 & 0 & 0 & 0 & -1 & 0 & -1 & 1 \\
0 & -1 & 0 & 0 & 1 & 0 & 0 & 0 & -1 & 0 \\
0 & 1 & 0 & 0 & 0 & 1 & 0 & -1 & 0 & 0
\end{array}\right)$}

\newsavebox{\sameAllsix}
\savebox{\sameAllsix}{\scriptsize $\left(\begin{array}{rrrrrrrrrr}
0 & 0 & 0 & 1 & 0 & 0 & 1 & 1 & 0 & 0 \\
0 & 1 & -1 & 1 & -1 & 0 & 0 & 0 & 0 & 0 \\
0 & 1 & 1 & 0 & 0 & 0 & 0 & 0 & 1 & 1 \\
0 & 0 & 0 & 0 & 0 & -1 & -1 & 0 & 0 & 1 \\
1 & 0 & 0 & 0 & 1 & 0 & 0 & -1 & 0 & 0 \\
1 & 0 & 0 & 0 & 0 & -1 & 0 & 0 & 1 & 0
\end{array}\right)$}

\newsavebox{\sameAllseven}
\savebox{\sameAllseven}{\scriptsize $\left(\begin{array}{rrrrrrrrrr}
-1 & -1 & 0 & 0 & 0 & 0 & 0 & 1 & -1 & 0 \\
-1 & 1 & 0 & 0 & -1 & 1 & 0 & 0 & 0 & 0 \\
0 & 0 & 1 & 0 & 1 & 0 & 0 & 0 & 0 & 1 \\
0 & 0 & 1 & -1 & 0 & 0 & 0 & 0 & 1 & 0 \\
0 & 0 & 0 & 1 & 0 & 0 & 1 & 1 & 0 & 0 \\
0 & 0 & 0 & 0 & 0 & -1 & -1 & 0 & 0 & 1
\end{array}\right)$}

\newsavebox{\sameAlleight}
\savebox{\sameAlleight}{\scriptsize $\left(\begin{array}{rrrrrrrrrr}
0 & 0 & 1 & 0 & 1 & 1 & 1 & 0 & 0 & 0 \\
0 & 0 & 1 & -1 & 0 & 0 & 0 & 0 & 1 & 0 \\
1 & 0 & 0 & 0 & 1 & 0 & 0 & -1 & 0 & 0 \\
0 & 1 & 0 & 1 & 0 & 0 & 0 & 0 & 0 & 1 \\
-1 & 0 & 0 & 0 & 0 & 0 & -1 & 0 & -1 & 1 \\
0 & 1 & 0 & 0 & 0 & 1 & 0 & -1 & 0 & 0
\end{array}\right)$}

\newsavebox{\sameAllnine}
\savebox{\sameAllnine}{\scriptsize $\left(\begin{array}{rrrrrrrrrr}
0 & 0 & 0 & 0 & 0 & 1 & 1 & 0 & 0 & -1 \\
0 & 0 & 1 & 0 & 1 & 0 & 0 & 0 & 0 & 1 \\
0 & 0 & 1 & -1 & 0 & 0 & 0 & 0 & 1 & 0 \\
1 & 0 & 0 & 0 & 1 & 0 & 0 & -1 & 0 & 0 \\
1 & 1 & 0 & 1 & 0 & 0 & 1 & 0 & 1 & 0 \\
0 & 1 & 0 & 0 & 0 & 1 & 0 & -1 & 0 & 0
\end{array}\right)$}

\newsavebox{\sameAllten}
\savebox{\sameAllten}{\scriptsize $\left(\begin{array}{rrrrrrrrrr}
1 & 0 & 0 & 0 & 0 & -1 & 0 & 0 & 1 & 0 \\
-1 & 0 & 0 & -1 & -1 & 0 & -1 & 0 & 0 & 0 \\
0 & 0 & 1 & 0 & 1 & 0 & 0 & 0 & 0 & 1 \\
0 & 0 & 1 & 0 & 0 & 0 & 1 & 1 & 1 & 0 \\
0 & 1 & 0 & 1 & 0 & 0 & 0 & 0 & 0 & 1 \\
0 & 1 & 0 & 0 & 0 & 1 & 0 & -1 & 0 & 0
\end{array}\right)$}

\newsavebox{\sameAlleleven}
\savebox{\sameAlleleven}{\scriptsize $\left(\begin{array}{rrrrrrrrrr}
0 & 0 & 0 & 0 & 0 & 1 & 1 & 0 & 0 & -1 \\
0 & 1 & 0 & 1 & 0 & 0 & 0 & 0 & 0 & 1 \\
0 & 1 & 0 & 0 & -1 & 0 & 0 & 0 & 1 & 0 \\
1 & 0 & 0 & 0 & 1 & 0 & 0 & -1 & 0 & 0 \\
1 & 0 & -1 & 1 & 0 & -1 & 0 & 0 & 0 & 0 \\
0 & 0 & 1 & 0 & 0 & 0 & 1 & 1 & 1 & 0
\end{array}\right)$}

\newsavebox{\sameAlltwelve}
\savebox{\sameAlltwelve}{\scriptsize $\left(\begin{array}{rrrrrrrrrr}
0 & 0 & 0 & 1 & 0 & 0 & 1 & 1 & 0 & 0 \\
0 & 1 & 0 & 1 & 0 & 0 & 0 & 0 & 0 & 1 \\
-1 & 0 & 1 & 0 & 0 & 0 & 0 & 1 & 0 & 1 \\
0 & 0 & 1 & 0 & 1 & 1 & 1 & 0 & 0 & 0 \\
0 & -1 & 0 & 0 & 1 & 0 & 0 & 0 & -1 & 0 \\
1 & 0 & 0 & 0 & 0 & -1 & 0 & 0 & 1 & 0
\end{array}\right)$}

\newsavebox{\sameAllthirteen}
\savebox{\sameAllthirteen}{\scriptsize $\left(\begin{array}{rrrrrrrrrr}
0 & 0 & 0 & 1 & 0 & 0 & 1 & 1 & 0 & 0 \\
1 & 1 & -1 & 1 & 0 & 0 & 0 & -1 & 0 & 0 \\
0 & 0 & 1 & 0 & 1 & 0 & 0 & 0 & 0 & 1 \\
0 & 0 & 0 & 0 & 0 & -1 & -1 & 0 & 0 & 1 \\
1 & 0 & 0 & 0 & 0 & -1 & 0 & 0 & 1 & 0 \\
0 & -1 & 0 & 0 & 1 & 0 & 0 & 0 & -1 & 0
\end{array}\right)$}

\newsavebox{\sameAllfourteen}
\savebox{\sameAllfourteen}{\scriptsize $\left(\begin{array}{rrrrrrrrrr}
0 & 0 & 0 & 1 & 0 & 0 & 1 & 1 & 0 & 0 \\
0 & 0 & -1 & 1 & 0 & 0 & 0 & 0 & -1 & 0 \\
0 & 0 & -1 & 0 & -1 & 0 & 0 & 0 & 0 & -1 \\
-1 & 1 & 0 & 0 & -1 & 0 & -1 & 0 & 0 & 1 \\
0 & 1 & 0 & 0 & 0 & 1 & 0 & -1 & 0 & 0 \\
1 & 0 & 0 & 0 & 0 & -1 & 0 & 0 & 1 & 0
\end{array}\right)$}

\newsavebox{\sameAllfifteen}
\savebox{\sameAllfifteen}{\scriptsize $\left(\begin{array}{rrrrrrrrrr}
0 & 0 & 0 & -1 & 0 & 0 & -1 & -1 & 0 & 0 \\
0 & 0 & -1 & 1 & 0 & 0 & 0 & 0 & -1 & 0 \\
-1 & 1 & 1 & 0 & 0 & 1 & 0 & 0 & 0 & 1 \\
0 & 0 & 0 & 0 & 0 & -1 & -1 & 0 & 0 & 1 \\
0 & 1 & 0 & 0 & -1 & 0 & 0 & 0 & 1 & 0 \\
1 & 0 & 0 & 0 & 1 & 0 & 0 & -1 & 0 & 0
\end{array}\right)$}

\newsavebox{\diffAllone}
\savebox{\diffAllone}{\scriptsize $\left(\begin{array}{rrrrrrrrr}
-1 & 0 & 0 & 1 & 0 & 0 & -1 & 1 & 0 \\
1 & 0 & 0 & 0 & 0 & -1 & 0 & 0 & 1 \\
0 & 1 & 0 & 0 & -1 & 0 & 0 & 0 & 1 \\
0 & 1 & -1 & 0 & 0 & 0 & 0 & 1 & 0 \\
0 & 0 & 1 & 0 & 1 & 1 & 1 & 0 & 0
\end{array}\right)$}

\newsavebox{\diffAlltwo}
\savebox{\diffAlltwo}{\scriptsize $\left(\begin{array}{rrrrrrrrr}
1 & 0 & 0 & 1 & 1 & 0 & 1 & 0 & 0 \\
0 & 0 & 0 & 1 & -1 & 0 & 0 & -1 & 0 \\
1 & 0 & 0 & 0 & 0 & -1 & 0 & 0 & 1 \\
0 & 1 & 1 & 0 & 0 & 1 & 1 & 0 & 1 \\
0 & 1 & -1 & 0 & 0 & 0 & 0 & 1 & 0
\end{array}\right)$}

\newsavebox{\diffAllthree}
\savebox{\diffAllthree}{\scriptsize $\left(\begin{array}{rrrrrrrrr}
0 & -1 & 0 & -1 & 0 & -1 & -1 & 0 & 0 \\
0 & 0 & -1 & 1 & 0 & 0 & 0 & 0 & -1 \\
0 & -1 & 1 & 0 & 0 & 0 & 0 & -1 & 0 \\
1 & 0 & 0 & 0 & 0 & -1 & 0 & 0 & 1 \\
1 & 0 & 0 & 0 & -1 & 0 & 1 & 1 & 0
\end{array}\right)$}

\newpage
\section{Introduction}

A class $\mathcal{M}$ of matroids is \emph{polynomial-time recognizable} if there is a polynomial $f(x)$ and an algorithm that determines after at most $f(n)$ rank evaluations, whether or not a fixed $n$-element matroid $M$ is in $\mathcal{M}$. Seymour (1981a) proved that the class of binary matroids is not polynomial-time recognizable, although the class of graphic matroids is. Later, Truemper (1982b) proved that the class of regular matroids is polynomial-time recognizable. Geelen and Mayhew independently showed that if a signed-graphic matroid is given by a rank oracle, then it is not polynomial-time recognizable. Since every signed-graphic matroid is a dyadic matroid, it follows that the class of dyadic matroids is not polynomial-time recognizable.

A matrix $A$ is \emph{binet} if $[I\mid A]$ is row-equivalent to an incidence matrix of a signed graph. In his doctoral thesis \cite{M}, Musitelli presents a first polynomial-time algorithm for recognizing binet matrices. In particular, his algorithm inputs a matrix $A$ that is potentially binet, and tries to reverse-engineer (using only row operations) the matrix into a signed-graphic incidence matrix. If $A$ is binet, then his algorithm finds a sequence of row operations bringing the matrix into a signed-graphic incidence matrix; otherwise, his algorithm determines that $A$ is not binet. 

However, we saw a number of concerns in Musitelli's work. \eat{Not all signed-graphic representations of a matroid are row-equivalent, but }Musitelli considers only row operations, and not column scaling, in order to recover a signed graph matrix from a binet one. Moreover, several different signed graph matrices can have the same row-reduced standard representation, but Musitelli does not mention which signed graph his algorithm returns, given a binet matrix.

This research was motivated by the desire to investigate the importance of column scaling in relating two signed-graphic representations of the same matroid.

\vskip-0.2in
\subsection{Our approach}

We used the Sage Mathematics software \cite{sage} to generate many examples of signed-graphic matroids and their signed-graphic representations. In particular, we first generated all non-isomorphic 3-connected uniquely $\mathbb{D}$-representable matroids of size up to 10 that could potentially be signed-graphic; then for each of these matroids, we computed all its non-isomorphic signed-graphic representations, and used Sage Notebook to draw the corresponding signed graphs. Our algorithms are discussed in Sections \ref{Section:alg_matroid} and \ref{Section:alg_reps}.\eat{; the Sage implementation of our algorithm can be found in Appendix \ref{App:Sage}.}

Using our Sage code, we were able to generate many interesting examples of signed graphs. Out of the 69 dyadic matroids we generated, 13 matroids had signed-graphic representations such that each pair of representations is row-equivalent; 39 had signed-graphic representations such that each pair is not row-equivalent; and the remaining 17 had signed-graphic representations such that some pairs are row-equivalent and some are not. Appendix \ref{App:SignedGraphs} provides an example of each kind.

In Section \ref{Section:Matroid-preserving_operations} we discuss an interesting matroid-preserving operation on a signed graph, which we call the \emph{cylinder flip}, that relates certain pairs of signed-graphic representations of the same matroid. Then in Section \ref{Section:Restricting_column_scaling} we present the Brylawski-Lucas Theorem, and illustrate how restricting column scaling to the factors $\pm 1$ weakens the theorem.

\vspace{-8pt}

\section{Definitions and notation}

We assume familiarity with matroid theory and graph theory, and refer to \cite{oxley} for any concepts and notation that remain undefined.

\eat{\subsection{Extensions and coextensions}
If a matroid $M$ is obtained from a matroid $N$ by deleting a non-empty subset $T$ of $E(N)$, then $N$ is called an \emph{extension} of $M$. In particular, if $|T|=1$, then $N$ is a \emph{single-element extension} of $M$. If $N^*$ is an extension of $M^*$, then $N$ is called a \emph{coextension} of $M$.}

\subsection{Signed graphs}

The tuple $\Omega = (G, \Sigma)$, where $G=(V, E)$ is a graph and $\Sigma \subseteq E$, is a \emph{signed graph}, the edges in $\Sigma$ being the \emph{negative} edges. The \emph{sign} of a path or cycle in $G$ is \emph{positive} if there is an even number of negative edges in it, and \emph{negative} otherwise.

\subsubsection{Signed graph matrix}
A signed graph is \emph{oriented} when each end of each edge is given a direction, so that in a positive edge the ends are both directed from one endpoint to the other, and in a negative edge either both ends are directed outward, to their own vertices, or both are directed inward, away from their vertices. Note that an oriented signed graph is a bidirected graph.

Given an orientation of a signed graph $\Omega$, its \emph{incidence matrix} is a $|V| \times |E|$ matrix over $GF(3)$, with a row for each vertex and a column for each edge, whose $ij$th entry is
	\[
	a_{ij} = \begin{cases}
	0 & \textrm{if vertex $i$ and edge $j$ are not adjacent, or if edge $j$ is a positive loop in $\Omega$}\\
	1 & \textrm{if edge $j$ is oriented into vertex $i$}\\
	-1& \textrm{if edge $j$ is oriented out of vertex $i$, or if edge $j$ is a negative loop incident with vertex $i$ in $\Omega$}\\
	\eat{0 & \textrm{if edge $j$ is a positive loop in $\Omega$}\\
	-1& \textrm{if edge $j$ is a negative loop incident with vertex $i$ in $\Omega$}\\}
	\end{cases}
	\]

\subsubsection{Binet matrices}

A matrix $A$ is called a \emph{binet matrix} if there exist both an integral matrix $M$ of full rank satisfying
    \[ \sum\limits_{i=1}^n |M_{ij}| \leq 2 \qquad \textrm{for any column index $j$} \]
and a basis $B$ of it such that $M = [B\mid N]$ (up to column permutation) and $A = B^{-1}N$. In other words, $A$ is binet if $[I\mid A]$ is row-equivalent to the incidence matrix of a bidirected graph.\footnote{Musitelli's definition of the node-edge incidence matrix of a bidrected graph is different from ours in that he distinguishes between half-edges and loops. For our purposes, this distinction is irrelevant and a mere semantic difference.} Because a signed graph matrix is the incidence matrix of a bidirected graph, the signed graph matrices row-reduced into standard form are a subset of binet matrices.

\subsection{Partial fields} A \emph{partial field} is a pair $\mathbb{P}=(R, G)$, where $R$ is a commutative ring, and $G \subseteq R^*$ is a subgroup such that $-1 \in G$. If $p \in G \cup \{ 0\}$, then we say $p$ is an \emph{element} of $\mathbb{P}$, and write $p \in \mathbb{P}$.  

\subsubsection{Matrix representation over a partial field}
Let $A$ be a matrix over $R$ having $r$ rows. Then $A$ is a \emph{weak $\mathbb{P}$-matrix} if, for each $r \times r$ submatrix $D$ of $A$, we have $\det(D) \in G \cup \{ 0 \}$.

A weak $\mathbb{P}$-matrix $A$ is a \emph{matrix representation} of a matroid $M$ on elements $E$ if for every $r$-subset $X \subseteq E$, the submatrix $A[X]$ of $A$ (indexed by the elements in $X$) has the property that
\[ \det(A[X]) \neq 0 \qquad\Leftrightarrow\qquad \textrm{$X$ is a basis of $M$}, \]
and there exists at least one $r$-subset $B \subseteq E$ such that $\det(A[B]) \neq 0$. If such a matrix representation exists for a matroid $M$, then we say $M$ is \emph{representable} over $\mathbb{P}$.

\eat{We say a matroid $M = (E, \mathcal{B})$ is \emph{representable} over a partial field $\mathbb{P}=(R, G)$ if every entry of its matrix representation is in the ring $R$ and all nonzero subdeterminants are in $G$.}

\subsection{Dyadic matroids}
The \emph{dyadic partial field} is
	\[ \mathbb{D} := (\mathbb{Q}, \{\pm 2^k : k \in \mathbb{Z} \}). \]
A \emph{dyadic matrix} is a matrix over $\mathbb{Q}$ all of whose nonzero subdeterminants are in $\{\pm 2^i : i \in \mathbb{Z} \}$. A matroid is \emph{dyadic} if it is representable over $\mathbb{D}$. The following theorem provides equivalent definitions for a dyadic matroid:

\eat{Dyadic matroids are a natural analogue of totally unimodular matrices.
A matroid $M$ is \emph{dyadic} if there is a dyadic matrix $A$ such that $M \cong M[A]$.}

\begin{thm}[Whittle, 1997] 
The following are equivalent for a matroid $M$.
\begin{enumerate}[\textnormal{(}i\textnormal{)}]
\item $M$ is dyadic.
\item $M$ is representable over $GF(3)$ and $GF(5)$.
\item $M$ is representable over $GF(p)$ for all odd primes $p$.
\item $M$ is representable over $GF(3)$ and $\mathbb{Q}$.
\item $M$ is representable over $GF(3)$ and $\mathbb{R}$.
\item $M$ is representable over $GF(3)$ and $GF(q)$ where $q$ is an odd prime power such that $q \equiv 2 \mod 3$.
\end{enumerate}
\end{thm}

Dyadic matroids are a natural analogue of regular matroids. Notably, Seymour's Decomposition Theorem for totally unimodular matrices characterizes the regular matroids as the 3-sums of graphic matroids, cographic matroids, and a certain 10-element matroid. In a similar spirit, as graphic matroids are to regular matroids, the signed-graphic matroids are a fundamental subclass of dyadic matroids.

\subsection{Signed-graphic matroids}
A matroid is \emph{signed-graphic} if it can be represented by an incidence matrix of a signed graph. Since the incidence matrix of a signed graph is a matrix over $GF(3)$ whose columns have at most 2 nonzero entries each, it follows that every signed-graphic matroid is dyadic. We say two signed-graphic representations are \emph{row-equivalent} if one can be obtained from the other via a sequence of elementary row operations.

\subsubsection{Circuits and bases}
The circuits of a signed graph are precisely the positive cycles, pairs of negative cycles that meet exactly in a vertex, and pairs of disjoint negative cycles, together with any path connecting them.

By a slight abuse of definition, we call an independent set of edges a \emph{forest} of the signed graph. Since every dependent set has a connected component containing a positive cycle or two negative cycles, the forests of a signed graph are precisely the subgraphs in which every connected component contains at most one negative cycle (and no positive cycles). A basis of the signed-graphic matroid corresponds to a \emph{spanning forest} of the signed graph.

\subsubsection{Resigning edges of a signed-graphic representation.}
For two sets $S$, $T$ we define \[ S\Delta T := \{ s \in S \mid s \notin T \} \cup \{t \in T \mid t \notin S \}.\] If $G = (V,E)$ is a graph, and $S \subseteq V$, then $\delta(S) := \seq*{e =uv \in E \mid u \in S, v \notin S }$.

Let $\Omega = (G, \Sigma)$ be a signed graph. For a vertex $v$, we can \emph{resign} $\Omega$ around $v$ by replacing $\Sigma$ by $\Sigma \Delta \delta(v)$. Resigning edges across a vertex set corresponds to row scaling by -1, and hence preserves the corresponding matroid $M(\Omega)$. We state this in a lemma:

\begin{lem}\label{Lem:resigning}Let $\Omega$ be the signed-graphic representation of a matroid $M$. For a vertex $v \in V(\Omega)$, let $\Omega'$ be obtained by resigning $\Omega$ around $v$. Then $M(\Omega) \cong M(\Omega')$.\end{lem}

A pair of vertices $\{u, v \}$ is called a \emph{blocking pair} if there exists a resigning of the edges such that every negative cycle meets $u$ or $v$.

\eat{
A $\Theta$-graph is shown in Fig. \ref{theta_graph}.

\begin{figure}\caption{\label{theta_graph}}\end{figure}

A \emph{biased graph} $\Omega = (G, \Psi)$ consists of a graph $G$ and a collection $\Psi$ of cycles of $G$, called \emph{balanced}, such that if $C_1, C_2 \in \Psi$ and $G[C_1\cup C_2]$ is a $\Theta$-graph, then the third cycle in $G[C_1 \cup C_2]$ is also in $\Psi$.

If every cycle of $G$ is in $\Psi$, then $\Omega$ is \emph{balanced}.

A \emph{signed graph} is obtained from a graph by placing a positive or negative sign on each edge. If $G$ is a signed graph, then one obtains a biased graph $(G, \Psi)$ by taking a cycle of $G$ to be balanced if it contains an even number of negative edges.

A matroid that is isomorphic to the bias matroid of a signed graph, that is, to the bias matroid of such a biased graph, is called a \emph{signed-graphic matroid}.

If $G = (V, E)$ is a graph, and $S \subseteq V$, then $\delta(S) := \{ e = uv \in E: u \in S, v \notin S \}$. \emph{Resigning} in $V$ by $p \in \mathbb{P}^*$ is the process of multiplying each edge in $\delta^+(V)$ (outgoing edge in the cut) with $p$ and multiplying each edge in $\delta^-(V)$ (incoming edge inthe cut) by $p^{-1}$. This does not change the matroid.}


\subsection{Cylinder graph}

\subsubsection{Contractible edges on a cylinder}
Given a \eat{cyclic}graph $G$ embedded on the cylinder, let $C$ be a cycle of $G$. We say $C$ is \emph{contractible} if we can repeatedly contract the edges of $C$ so that $C$ becomes a single vertex, and \emph{noncontractible} otherwise.

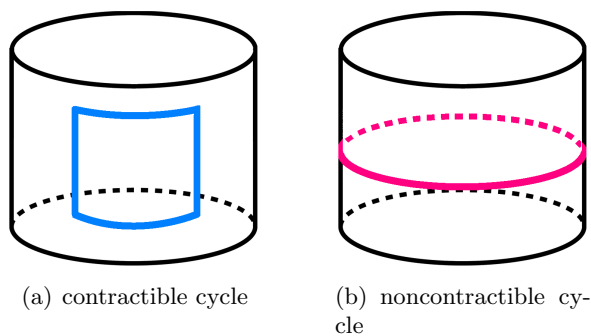
\renewcommand\thesubfigure{(\alph{subfigure})}
\begin{figure}[H]
    \centering
    \subfigure[contractible cycle]{\includegraphics[scale=0.14]{cylinder_contractable}}
    \hspace{20pt}
    \subfigure[noncontractible cycle]{\includegraphics[scale=0.14]{cylinder_uncontractable}}
\end{figure}

\subsubsection{Cylinder graph}

We define the \emph{cylinder graph} of a signed graph $\Omega = (G, \Sigma)$ in the following way. Suppose $\Omega$ has two blocking pairs $\{s_1, s_2\}$ and $\{t_1, t_2\}$ such that $\{s_1, s_2, t_1, t_2\}$ form a 4-vertex cut. Because $\{t_1, t_2\}$ is a blocking pair, there exists a resigning of the edges such that every negative cycle meets $t_1$ or $t_2$. We use this resigning to embed $\Omega$ on a cylinder such that all positive cycles are contractible and all the negative cycles are noncontractible.

\begin{figure}[H]
\includegraphics[scale=0.13]{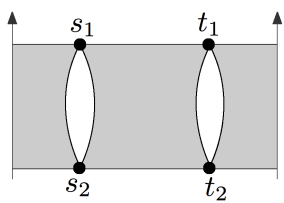}
\caption{$\Omega$ embedded on a cylinder, where the edges only cross inside the gray area. The cylinder has been cut open along the arrows.}
\end{figure}

If we split each of the blocking pair vertices into two (i.e., $s_i$ into $s_i$ and $s_i'$, and $t_i$ into $t_i$ and $t_i'$) as shown in the figure below, then $\Omega$ becomes divided into two components (each of which is not necessarily connected), call them $H_1$ and $H_2$.
\begin{figure}[H]
\includegraphics[scale=0.13]{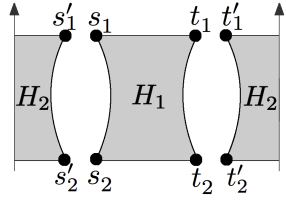}
\caption{$\Omega$ after splitting each of the blocking pair vertices into two.}
\end{figure}

Now consider the original embedding of $\Omega$ on the cylinder, before the blocking pair vertices were split. Suppose there exists a resigning $\Sigma'$ of the edges such that every negative edge is incident with either $t_1$ or $t_2$, and the negative edges are either all contained in $H_1$ or all contained in $H_2$. Then such embedding of the signed graph $(G, \Sigma')$ on the cylinder is called a \emph{cylinder graph on $\{s_1, s_2 ; t_1, t_2 \}$}.

We note that our definition of cylinder graphs is slightly different from the definition in \cite{vZ2}; for example, our definition does not require planarity.


\section{Generating dyadic matroids}\label{Section:alg_matroid}

To generate all non-isomorphic 3-connected uniquely $\mathbb{D}$-representable matroids of size up to $k$ that could be potentially signed-graphic, our algorithm incrementally computes the simple dyadic extensions and cosimple dyadic coextensions of each $M \in \{F_7^-, (F_7^-)^* \}$. More explicitly:

\vphantom{x}
\begin{center}
\begin{mdframed}
\begin{tt}
\noindent{Generate-Matroids \{}\\
\hspace*{0.2in}$\mathcal{M} \leftarrow \{ F_7^-, (F_7^-)^* \}$\\
\hspace*{0.2in}for $i$ from 1 to $(k-7)$: \\
\hspace*{0.4in} for each $M$ in $\mathcal{M}$: \{\\
\hspace*{0.6in} $\mathcal{N}\leftarrow$ all simple dyadic single-element extensions and cosimple dyadic\\
\hspace*{1.01in} single-element coextensions of $M$\\
\hspace*{0.6in} for each $N \in \mathcal{N}$:\\
\hspace*{0.8in} Add $N$ to $\mathcal{M}$ if it is not isomorphic to any previously added matroid.\\
\hspace*{0.6in}\}\\
\hspace*{0.2in}return $\mathcal{M}$\\
\}
\end{tt}
\end{mdframed}
\end{center}
\vphantom{x}

In the following sections, we explain the correctness of our algorithm and why these matroids are desirable for our purposes.


\subsection{Unique $\mathbb{D}$-representation}
Since we are interested in the row operations and column scalings required to get from one signed-graphic representation to another of the same matroid, we would like to generate only the matroids that have unique dyadic representations.

By \cite[6.3.4]{vZ}, $P_8$, $F_7^-$, and $(F_7^-)^*$ are uniquely $\mathbb{D}$-representable; and moreover, they are a stabilizer for $\mathbb{D}$. So the simple extensions and cosimple co-extensions over $\mathbb{D}$ of $P_8$, $F_7^-$, and $(F_7^-)^*$ are uniquely $\mathbb{D}$-representable, by the following proposition (\cite[14.8.2]{oxley}):
\begin{prop}Let $\mathbb{F}$ be a field and $N$ be an $\mathbb{F}$-stabilizer for the class $\mathcal{M}$ of $\mathbb{F}$-representable matroids. If $N$ is uniquely representable over $\mathbb{F}$, then so is every 3-connected matroid in $\mathcal{M}$ that has an $N$-minor.\end{prop}

However, we are not interested in the extensions and coextensions of $P_8$ because matroids with a $P_8$ minor are not signed-graphic. 


\subsection{3-connectedness}
Since every 2-connected matroid can be written in terms of 2-sums of 3-connected matroids (\cite[Cunningham and Edmonds 1980, Seymour 1981b]{oxley}), it is natural to investigate only the 3-connected matroids.

\begin{lem}\label{deletion} If $M \backslash e$ is 3-connected and $M$ is not 3-connected, then $e$ must be a parallel pair or a loop.\end{lem}
\begin{proof}Let $(A, B)$ be a 2-separation of $M$ with $e \in B$. Then by definition of 2-separation,
	\[ r(A) + r(B) - r(A \cup B) \leq 1 \]
	\[ |A|, |B| \geq 2. \]
If we delete $e$, then $r(B\backslash e) - r(A \cup B\backslash e)$ will either stay the same or decrease by 1, and so
	\[ r(A)+r(B\backslash e) - r(A \cup B\backslash e) \leq 1.\]
But because $M \backslash e$ is 3-connected, $(A, B\backslash e)$ cannot be a 2-separation of $M \backslash e$. So $|B| = 2$, i.e., $B = \{e, f \}$ for some element $f$. If $\{ e, f \}$ is a series pair, then $M \backslash e$ will have 1-separation, contradicting the assumption that $M\backslash e$ is 3-connected. Therefore, $e$ must be a parallel pair or a loop.\end{proof}

\begin{lem}If $M / e$ is 3-connected and $M$ is not 3-connected, then $e$ must be a series pair or a coloop.\end{lem}
\begin{proof} Since 3-connectedness is closed under duality, and contraction in $M$ corresponds to deletion in $M^*$, applying Lemma \ref{deletion} to $M^*$ concludes the proof.
\end{proof}
Since the algorithm computes simple extensions and cosimple coextensions of $F_7^-$ and $(F_7^-)^*$, the algorithm never adds any parallel pairs, loops, series pairs, or coloops to a 3-connected matroid. So it follows from the above lemmas that the algorithm produces 3-connected matroids.

More generally, consider the following important result in matroid theory (Seymour 1980b):
\begin{thm}[Splitter Theorem] Let $N$ be a non-empty, connected, simple, cosimple minor of a 3-connected matroid $M$. Suppose that $N$ is neither a wheel nor a whirl. Then either $M = N$, or $M$ has a connected, simple, cosimple minor $M_1$ such that some single-element deltion or some single-element contraction of $M_1$ is isomorphic to $N$. Moreover, if $N$ is 3-connected, so too is $M_1$.
\end{thm}

Wheels and whirls are an exception to the Splitter theorem, but this is not a problem in our case since there are no wheels or whirls in matroids with $F_7^-$, $(F_7^-)^*$, or $P_8$ as a minor.

\begin{cor}Let $M$ and $N$ be 3-connected matroids such that $N$ is a minor of $M$ with at least four elements, and $N$ is neither a wheel nor a whirl. Then there is a sequence $M_0, M_1, \ldots, M_n$ of 3-connected matroids with $M_0 \cong N$ and $M_n = M$ such that $M_i$ is a single-element deletion or a single-element contraction of $M_{i+1}$ for all $i \in \{0, 1, \ldots, n-1 \}$. \end{cor}

Starting with $M_0 \in \{ F_7^-, (F_7^-)^* \}$, our algorithm generates a set containing all sequences $M_0, M_1, \ldots, M_n$ of 3-connected matroids such that $M_{j+1}$ is a single-element extension or coextension of $M_{j}$.

\eat{All dyadic matroids that are not near-regular have non-Fano, nonFano dual, or $P_8$ as a minor.}

\section{Generating signed-graphic representations}\label{Section:alg_reps} To generate all signed-graphic representations of a given dyadic matroid, we implemented Prof. Rudi Pendavingh's recursive function {\tt graphic\_rec} (originally written in C++) in Sage. The following nonrecursive pseudocode provides a simplified overview of the algorithmic approach, although the algorithm performs some clever optimizations to improve the running time.

\vphantom{x}
\begin{mdframed}
\begin{tt}
\noindent{graphic\_rec (input: dyadic matroid $M$)} \{ \\
\hspace*{0.2in}$\mathcal{A} \leftarrow \emptyset$\\
\hspace*{0.2in}$R \leftarrow$ matrix representation of $M$\\
\hspace*{0.2in}$X \leftarrow$ all dyadic extension columns of $M$ \\
\hspace*{0.2in}for each full-rank $r$-subset $S \subseteq X$: \{\\
\hspace*{0.4in} Row-reduce $[R \mid S] \rightarrow [R' \mid I]$ \\
\hspace*{0.4in} if $R'$ has $\leq 2$ nonzeroes per column:\\
    \hspace*{0.6in}Add $R'$ to $\mathcal{A}$\\
\hspace*{0.2in}\}\\
\hspace*{0.2in}Return $\mathcal{A}$\\
\}
\end{tt}
\end{mdframed}
\vphantom{x}

If $X$ is the set of all dyadic extension columns of a rank-$r$ matroid $M$, then the columns of $X$ correspond to a finite number of points in the ambient space. The algorithm tries to find an $r$-subset of these points $\vec{x}_1, \vec{x}_2, \ldots, \vec{x}_r \in X$ such that $\{ \vec{x}_1, \vec{x}_2, \ldots, \vec{x}_r\}$ form a basis and every original element of the matroid is spanned by at most two elements from this basis. If such a basis exists, then the matroid can be represented as $[I\mid A]$, where the columns of $A$ correspond to the original elements of the matroid, and each column of $A$ contains at most 2 nonzero entries. Then $A$ is a signed-graphic representation of $M$.

In each call to the recursive function, the algorithm grabs an appropriate new element $\vec{x}$ from $X$ and turns it into a basis element, then checks if the current (incomplete) set of basis elements can be extended into a basis with the desired properties. Essentially, the algorithm adds one edge at a time to grow a spanning forest, similar to Prim's greedy algorithm for growing a minimum spanning tree of a connected graph.

\subsection{Initial call to the recursive function}
Given a rank-$r$ dyadic matroid $M$ on $n$ elements, let $R$ be a dyadic representation of $M$, and $X$ the set of all dyadic extension columns of $M$. The algorithm calls the recursive function {\tt graphic\_rec} (described in the following section) with the initial parameters $k=1$, $R$, and $X$.

\subsection{The recursive function}

The {\tt graphic\_rec} function inputs an integer $k \geq 1$, and two matrices $R$ and $X$ of the form

\begin{footnotesize}
\begin{align*}
R = \begin{blockarray}{ccc}
 \begin{block}{c(c|c)}
     &  \phantom{xxxxxx} & \phantom{xxxxxxxxxx}        \\  
     & I_{k-1} & \\
     &&\\
     \cline{2-2}
    &&\mathop{\vphantom{\int}}^{\textstyle A}  \\  
     &0 & \\  
     &        &        \\  
 \end{block}
\end{blockarray}
\qquad\qquad\quad
X = \begin{blockarray}{ccc}
 \begin{block}{c(c|c)}
     & \phantom{xxxxxx}  &  \\  
     & I_{k-1} & \Big{*}
     &        &        \\
     &        &        \\\cline{2-4}
    &        &        \\  
     & 0 & \textrm{\scriptsize $\geq 1$ nonzeroes per column}\\  
     &        &        \\  
 \end{block}
\end{blockarray}
\end{align*}
\end{footnotesize}

where the submatrix $A$ satisfies the following property for $N=k-1$:
\begin{property}\label{condition}
If a column has more than $2$ nonzero entries in the first $N$ entries, then $\exists \ell \geq N+1$ such that the $\ell$th entry of that column is nonzero.
\end{property}

\eat{\begin{equation}\label{condition}
\textrm{if there exists a column index $j$ such that } \sum\limits_{i=1}^{k-1} |A_{ij}| > 2, \quad\textrm{then $\exists \ell \geq k$ such that $A_{\ell j} \neq 0$.}
\end{equation}}

We require this condition on $R$ for the following reason: If a column of $A$ has more than 2 nonzero entries in the already processed rows, then a future pivot is needed to repair that column so that it has $\leq 2$ nonzero entries. If the property is not satisfied, then the column cannot be repaired and hence $R$ cannot be signed-graphic.

The first $k-1$ columns of $R$ and of $X$ correspond to the incomplete set of basis elements that the algorithm has already found prior to the current recursive call. Furthermore, this incomplete set of basis elements corresponds to a forest $F$, which we would like to grow into a spanning forest. In choosing the next element to add to the matroid, the algorithm considers three cases: it can grow an existing partial component of $F$ into a larger one, add a negative loop, or declare the existing components to be done and start a new component.


\subsubsection{Growing an existing partial component} \label{subsubsection:growing_existing_component}

Suppose there exists a column $\vec{x} \in X$ such that it has 1 nonzero entry among the first $k$ rows, and exactly one nonzero entry among the remaining $r-k$ rows. Then we may try to grow an existing component of the forest $F$. Let $t < k$ and $s \geq k$ be the unique indices such that $b:=\vec{x}[t]$ and $a:=\vec{x}[s]$ are nonzero.\eat{, i.e.,
\[ x = 
\begin{blockarray}{cc}
\begin{block}{c[c]}
 &   \\
t& b \\
 &  \\
 k&  \\
 s & a  \\
 &  \\
\end{block}
\end{blockarray}
\]}

If $R' := R \cup \{\vec{x}\}$ is obtained by inserting $\vec{x}$ between the $(k-1)$th and $k$th columns in $R$, then
\begin{footnotesize}
\[
R' = \begin{blockarray}{cccc}
&&k\\
 \begin{block}{c(c|c|c)}
     & \phantom{xxxxxx}& &\phantom{xxxxxxxxxx} \\  
     & I_{k-1}& &
     &        &&        \\
     &        &&        \\\cline{2-2}
    &        &\mathop{\vphantom{\int}}^{\displaystyle\vec{x}}&\mathop{\vphantom{\int}}^{\textstyle A} \\  
     &0 &&\\  
     &        &&        \\  
 \end{block}
\end{blockarray}
\;=\begin{blockarray}{cccc}
&&k\\
 \begin{block}{c(c|c|c)}
     & \phantom{xxxxxx}& & \phantom{xxxxxxxxxx}\\
     \cline{2-4}  
     t& \vec{e}_t & b & \vec{y}\\
     \cline{2-4}
     &        &&        \\
    k&        &&        \\  
     \cline{2-4}
     s&0\cdots0 &a& \vec{z}\\  
     \cline{2-4}
     &        &&        \\  
 \end{block}
\end{blockarray}
\]
\end{footnotesize}

where $\vec{y}$ (resp. $\vec{z}$) is the $t$th (resp. $s$th) row of $A$.  To turn $\vec{x}$ into a basis element, the algorithm pivots over $\vec{x}[s]$, then swaps the $k$th row and $s$th row of $R$:
\begin{footnotesize}
\begin{align*}
\begin{array}{ccc}
\begin{blockarray}{cccc}
&&k\\
 \begin{block}{c(c|c|c)}
     & \phantom{xxxxxx}& & \phantom{xxxxxxxxxx}\\
     \cline{2-4}  
     t& \vec{e}_t & b & \vec{y}\\
     \cline{2-4}
     &        &&        \\
    k&        &&        \\  
     \cline{2-4}
     s&0\cdots0 &a& \vec{z}\\  
     \cline{2-4}
     &        &&        \\  
 \end{block}
\end{blockarray}
\quad\rightarrow
&\begin{blockarray}{cccc}
&&k\\
 \begin{block}{c(c|c|c)}
     & \phantom{xxxxxx}& &\phantom{xxxxxxxxxx}\\
     \cline{2-4}  
     t& \vec{e}_t & 0 & \vec{y}- ba^{-1}\vec{z}\\
     \cline{2-4}
     &        &&        \\
    k&        &&        \\  
     \cline{2-4}
     s&0\cdots0 &1& a^{-1}\vec{z}\\  
     \cline{2-4}
     &        &&        \\  
 \end{block}
\end{blockarray}
\quad\rightarrow
&\begin{blockarray}{cccc}
&&k\\
 \begin{block}{c(c|c|c)}
     & \phantom{xxxxxx}& & \phantom{xxxxxxxxxx}\\
     \cline{4-4}  
     t& I_{k-1} & 0 & \vec{y}- ba^{-1}\vec{z}\\
     \cline{4-4}
     &        &&        \\
     \cline{2-4}
     k&0\cdots0 &1& a^{-1}\vec{z}\\  
     \cline{2-4}
     s&&&\\
     &        &&        \\  
 \end{block}
\end{blockarray}
\\
\textrm{(1) } R' & \textrm{(2) Pivot over $\vec{x}[s] = R'_{sk}$} & \textrm{(3) Swap rows $k$ and $s$}
\end{array}\end{align*}
\end{footnotesize}

Now $R'$ is of the form
\begin{footnotesize}
\[
R' = \begin{blockarray}{ccc}
 \begin{block}{c(c|c)}
     &  \phantom{xxxxxxxxxx} & \phantom{xxxxxxxxxx}        \\  
     & I_{k} & \\
     &&\\
     \cline{2-2}
    &&\mathop{\vphantom{\int}}^{\textstyle A'}  \\  
     &0 & \\  
     &        &        \\  
 \end{block}
\end{blockarray}
\]
\end{footnotesize}

The algorithm then checks whether $A'$ satisfies property \ref{condition} for $N=k$. If $A'$ does not satisfy the property, then $R'$ can no longer be turned into a signed-graphic representation, so the algorithm prunes search by returning the current recursive function call.

Otherwise if $k=r$, then property \ref{condition} guarantees that $R' = [I \mid A]$ where every column of $A$ has $\leq 2$ nonzero entries, which implies that the representation is signed-graphic. So the algorithm stores $A$ in $\mathcal{A}$ and returns the current function call.

Otherwise if $k < r$, then the algorithm prepares for another recursive call. Without loss of generality, we may assume that $\vec{x}$ is the $k$th column of $X$, i.e.,

\begin{footnotesize}
\begin{align*}
X = \begin{blockarray}{cccc}
 &&k\\
 \begin{block}{c(c|c|c)}
     & \phantom{xxxxxx}  &  &\\  
     & I_{k-1} & & \Big{*}
     &        & &       \\
     &        &&       \\\cline{2-2}\cline{4-4}
    &        &\mathop{\vphantom{\int}}^{\displaystyle\vec{x}} &        \\  
     & 0 & &\textrm{\scriptsize $\geq 1$ nonzeroes per column}\\  
     &      &  &        \\  
 \end{block}
\end{blockarray}
= \begin{blockarray}{cccc}
&&k\\
 \begin{block}{c(c|c|c)}
     & \phantom{xxxxxx}  &  &\phantom{xxxxxxxxxxxxxxx} \\
     \cline{2-4}
     t& \vec{e}_t & b & \vec{u}\\
     \cline{2-4}
     &        & &       \\
     &        &&       \\
    k&        & &        \\
    \cline{2-4}
     s& 0\cdots0 & a& \vec{w}\\
     \cline{2-4}
     &      &  &        \\  
 \end{block}
\end{blockarray}
\end{align*}
\end{footnotesize}

First, the algorithm pivots $X$ over $\vec{x}[s]$, then swaps the $k$th and $s$th row of $X$:
\begin{footnotesize}
\[
\begin{array}{ccc}

\begin{blockarray}{cccc}
&&k\\
 \begin{block}{c(c|c|c)}
     & \phantom{xxxxxx}& & \phantom{xxxxxxxxxx}\\
     \cline{2-4}  
     t& \vec{e}_t & b & \vec{u}\\
     \cline{2-4}
     &        &&        \\
    k&        &&        \\  
     \cline{2-4}
     s&0\cdots0 &a& \vec{w}\\  
     \cline{2-4}
     &        &&        \\  
 \end{block}
\end{blockarray}
\quad\rightarrow
&\begin{blockarray}{cccc}
&&k\\
 \begin{block}{c(c|c|c)}
     & \phantom{xxxxxx}& &\phantom{xxxxxxxxxx}\\
     \cline{2-4}  
     t& \vec{e}_t & 0 & \vec{u}- ba^{-1}\vec{w}\\
     \cline{2-4}
     &        &&        \\
    k&        &&        \\  
     \cline{2-4}
     s&0\cdots0 &1& a^{-1}\vec{w}\\  
     \cline{2-4}
     &        &&        \\  
 \end{block}
\end{blockarray}
\quad\rightarrow
&\begin{blockarray}{cccc}
&&k\\
 \begin{block}{c(c|c|c)}
     & \phantom{xxxxxx}& & \phantom{xxxxxxxxxx}\\
     \cline{4-4}  
     t& I_{k-1} & 0 & \vec{u}- ba^{-1}\vec{w}\\
     \cline{4-4}
     &        &&        \\
     \cline{2-4}
     k&0\cdots0 &1& a^{-1}\vec{w}\\  
     \cline{2-4}
     s&&&\\
     &        &&        \\  
 \end{block}
\end{blockarray}
\\
\textrm{(1) } X & \textrm{(2) Pivot over $\vec{x}[s]=X_{sk}$} & \textrm{(3) Swap rows $k$ and $s$}
\end{array}\]
\end{footnotesize}

Then the algorithm throws away the columns in $X$ which do not have a nonzero entry below the $k$th row (i.e., the columns which are entirely spanned by $\{e_1, e_2, \ldots, e_k\}$). Finally, the algorithm recursively calls {\tt graphic\_rec} with the parameters $k+1$, $R'$, and $X$.


\subsubsection{Adding a negative loop}
Suppose there exists a column $\vec{x} \in X$ such that its only nonzero entry is on or below the $k$th row. 
Then we may try to add a negative loop to the forest $F$. Let $s \geq k$ be the unique index such that $a:=\vec{x}[s]$ is nonzero.

If $R' := R \cup \{\vec{x}\}$ is obtained by inserting $\vec{x}$ between the $(k-1)$th and $k$th columns in $R$, then
\begin{footnotesize}
\[
R' = \begin{blockarray}{cccc}
&&k\\
 \begin{block}{c(c|c|c)}
     & \phantom{xxxxxx}& &\phantom{xxxxxxxxxx} \\  
     & I_{k-1}& &
     &        &&        \\
     &        &&        \\\cline{2-2}
    &        &\mathop{\vphantom{\int}}^{\displaystyle\vec{x}}&\mathop{\vphantom{\int}}^{\textstyle A} \\  
     &0 &&\\  
     &        &&        \\  
 \end{block}
\end{blockarray}
\;=\begin{blockarray}{cccc}
&&k\\
 \begin{block}{c(c|c|c)}
     & \phantom{xxxxxx}& 0 & \phantom{xxxxxxxxxx}\\
     & I_{k-1} &\vdots  & \\
     &        &0&      \\\cline{2-2}
    k&        & & \\
     \cline{2-4}
     s&0\cdots0 &a& \vec{z}\\  
     \cline{2-4}
     &        &0&        \\  
 \end{block}
\end{blockarray}
\]
\end{footnotesize}

where $\vec{z}$ is the $s$th row of $A$.  To turn $\vec{x}$ into a basis element, the algorithm pivots over $\vec{x}[s]$, then swaps the $k$th row and $s$th row of $R$:

\begin{footnotesize}
\[
\begin{array}{ccc}

\begin{blockarray}{cccc}
&&k\\
 \begin{block}{c(c|c|c)}
     & \phantom{xxxxxx}& 0& \phantom{xxxxxxxxxx}\\
     & I_{k-1} & \vdots & \\
     &        &0&       \\\cline{2-2}
    k&        &&        \\
     \cline{2-4}
     s&0\cdots0 &a& \vec{z}\\  
     \cline{2-4}
     &        &0&        \\  
 \end{block}
\end{blockarray}
\quad\rightarrow
&\begin{blockarray}{cccc}
&&k\\
 \begin{block}{c(c|c|c)}
     & \phantom{xxxxxx}&0& \phantom{xxxxxxxxxx}\\
     & I_{k-1} & \vdots & \\
     &        &0&       \\\cline{2-2}
    k&        &&        \\
     \cline{2-4}
     s&0\cdots0 &1& a^{-1}\vec{z}\\  
     \cline{2-4}
     &        &0&        \\  
 \end{block}
\end{blockarray}
\quad\rightarrow
&\begin{blockarray}{cccc}
&&k\\
 \begin{block}{c(c|c|c)}
     & \phantom{xxxxxx}&0 & \phantom{xxxxxxxxxx}\\
     & I_{k-1} & \vdots & \\
     &        &0&      \\\cline{2-2}
     \cline{2-4}
     k&0\cdots0 &1& a^{-1}\vec{z}\\  
     \cline{2-4}
     s&&\\
     &        &0&        \\  
 \end{block}
\end{blockarray}
\\
\textrm{(1) } R' & \textrm{(2) Pivot over $\vec{x}[s] = R'_{sk}$} & \textrm{(3) Swap rows $k$ and $s$}
\end{array}\]
\end{footnotesize}

Now $R'$ is of the form

\begin{footnotesize}
\[
R' = \begin{blockarray}{ccc}
 \begin{block}{c(c|c)}
     &  \phantom{xxxxxxxxxx} & \phantom{xxxxxxxxxx}        \\  
     & I_{k} & \\
     &&\\
     \cline{2-2}
    &&\mathop{\vphantom{\int}}^{\textstyle A'}  \\  
     &0 & \\  
     &        &        \\  
 \end{block}
\end{blockarray}
\]
\end{footnotesize}

The algorithm then checks whether $A'$ satisfies property \ref{condition} for $N=k$. If $A'$ does not satisfy the property, then $R'$ can no longer be turned into a signed-graphic representation, so the algorithm prunes search by returning the current recursive function call.

Otherwise if $k=r$, then property \ref{condition} guarantees that $R' = [I \mid A]$ where every column of $A$ has $\leq 2$ nonzero entries, which implies that the representation is signed-graphic. So the algorithm stores $A$ in $\mathcal{A}$ and returns the current function call.

Otherwise if $k < r$, then the algorithm prepares for another recursive call. Similar to the previous case, the algorithm pivots $X$ over $\vec{x}[s]$, swaps the $k$th and $s$th row of $X$, then throws away the columns in $X$ which do not have a nonzero entry below the $k$th row. Finally, the algorithm recursively calls {\tt graphic\_rec} with the parameters $k+1$, $R'$, and $X$.


\subsubsection{Starting a new component}
Suppose we declare the existing components to be done, i.e., we have exhausted possibilities for growing already existing components and would like to start a new component. Then the algorithm throws away the columns in $X$ that have a nonzero entry in the first $k-1$ rows (i.e., the columns that connect to already existing components), and arbitrary chooses one of the remaining columns $\vec{x} \in X$ that has a nonzero entry on or below the $k$th row. Let $s\geq k$ be the smallest integer such that $a := \vec{x}[s] \neq 0$.

If $R' := R \cup \{\vec{x}\}$ is obtained by inserting $\vec{x}$ between the $(k-1)$th and $k$th columns in $R$, then
\begin{footnotesize}
\[
R' = \begin{blockarray}{cccc}
&&k\\
 \begin{block}{c(c|c|c)}
     & \phantom{xxxxxx}& &\phantom{xxxxxxxxxx} \\  
     & I_{k-1}& &
     &        &&        \\
     &        &&        \\\cline{2-2}
    &        &\mathop{\vphantom{\int}}^{\displaystyle\vec{x}}&\mathop{\vphantom{\int}}^{\textstyle A} \\  
     &0 &&\\  
     &        &&        \\  
 \end{block}
\end{blockarray}
\;=\begin{blockarray}{cccc}
&&k\\
 \begin{block}{c(c|c|c)}
     & \phantom{xxxxxx}& 0 & \phantom{xxxxxxxxxx}\\
     & I_{k-1} &\vdots  & \\
     &        &0&    \\\cline{2-2}
    k&        & & \\
     \cline{2-4}
     s&0\cdots0 &a& \vec{z}\\  
     \cline{2-4}
     &        &\vec{b}& B      \\  
 \end{block}
\end{blockarray}
\]
\end{footnotesize}

where $\vec{z}$ is the $s$th row of $A$, $B$ is the $(r-s)\times n$ submatrix of $A$ below $\vec{z}$, and $\vec{b}$ is the $(r-s)$-subvector of $\vec{x}$ below $\vec{x}[s]$.  To turn $\vec{x}$ into a basis element, the algorithm pivots over $\vec{x}[s]$, then swaps the $k$th row and $s$th row of $R$:

\begin{footnotesize}
\[
\begin{array}{ccc}

\begin{blockarray}{cccc}
&&k\\
 \begin{block}{c(c|c|c)}
     & \phantom{xxxxxx}& 0 & \phantom{xxxxxxxxxx}\\
     & I_{k-1} &\vdots  & \\
     &        &0&    \\\cline{2-2}
    k&        & & \\
     \cline{2-4}
     s&0\cdots0 &a& \vec{z}\\  
     \cline{2-4}
     &        &\vec{b}& B      \\  
 \end{block}
\end{blockarray}
\quad\rightarrow
&\begin{blockarray}{cccc}
&&k\\
 \begin{block}{c(c|c|c)}
     & \phantom{xxxxxx}&0& \phantom{xxxxxxxxxx}\\
     & I_{k-1} & \vdots & \\
     &        &0&       \\\cline{2-2}
    k&        &&        \\
     \cline{2-4}
     s&0\cdots0 &1& a^{-1}\vec{z}\\  
     \cline{2-4}
     &        &0& B- a^{-1} \vec{b}\vec{z}       \\  
 \end{block}
\end{blockarray}
\quad\rightarrow
&\begin{blockarray}{cccc}
&&k\\
 \begin{block}{c(c|c|c)}
     & \phantom{xxxxxx}&0 & \phantom{xxxxxxxxxx}\\
     & I_{k-1} & \vdots & \\
     &        &0&      \\\cline{2-2}
     \cline{2-4}
     k&0\cdots0 &1& a^{-1}\vec{z}\\  
     \cline{2-4}
     s&&\\
     \cline{4-4}
     &        &0&B-a^{-1}\vec{b}\vec{z}        \\  
 \end{block}
\end{blockarray}
\\
\textrm{(1) } R' & \textrm{(2) Pivot over $\vec{x}[s] = R'_{sk}$} & \textrm{(3) Swap rows $k$ and $s$}
\end{array}\]
\end{footnotesize}

Now $R'$ is of the form

\begin{footnotesize}
\[
R' = \begin{blockarray}{ccc}
 \begin{block}{c(c|c)}
     &  \phantom{xxxxxxxxxx} & \phantom{xxxxxxxxxx}        \\  
     & I_{k} & \\
     &&\\
     \cline{2-2}
    &&\mathop{\vphantom{\int}}^{\textstyle A'}  \\  
     &0 & \\  
     &        &        \\  
 \end{block}
\end{blockarray}
\]
\end{footnotesize}

The algorithm then checks whether $A'$ satisfies property \ref{condition} for $N=k$. If $A'$ does not satisfy the property, then $R'$ can no longer be turned into a signed-graphic representation, so the algorithm prunes search by returning the current recursive function call.

Otherwise if $k=r$, then property \ref{condition} guarantees that $R' = [I \mid A]$ where every column of $A$ has $\leq 2$ nonzero entries, which implies that the representation is signed-graphic. So the algorithm stores $A$ in $\mathcal{A}$ and returns the current function call.

Otherwise if $k < r$, then the algorithm prepares for another recursive call. Similar to the previous cases, the algorithm pivots $X$ over $\vec{x}[s]$, swaps the $k$th and $s$th row of $X$, then throws away the columns in $X$ which do not have a nonzero entry below the $k$th row. Finally, the algorithm recursively calls {\tt graphic\_rec} with the parameters $k+1$, $R'$, and $X$.

\section{Matroid-preserving operations on a signed-graphic representation}\label{Section:Matroid-preserving_operations}

For certain pairs of signed-graphic representations $A$, $B$ of the same matroid, we can get from $A$ to $B$ via a sequence of matroid-preserving operations. One such operation is the resigning of edges across a vertex set (see Lemma \ref{Lem:resigning}). We present another matroid-preserving operation called the \emph{cylinder flip}, which is reminiscent of the Whitney flip \cite[Ch. 5]{pivotto}. In section \ref{subsection:flip_example}, we present an example where two signed-graphic representations of the same matroid are related via a sequence of resigning edges and cylinder flips.

\eat{\subsection{Resigning edges}
By Lemma \ref{Lem:resigning}, resigning edges across a vertex set preserves the corresponding matroid.
\eat{If you want to resign across a cut, you can resign along any cut. Every edge inside the part will be resigned twice. Only the edges with one edge inside cut and the other edge outside cut will be changed.}}

\subsection{Cylinder flip}

\begin{df}\eat{We define the \emph{cylinder flip} operation on a cylinder graph as follows.} For a given cylinder graph $\Omega$ on the blocking pair vertices $\{s_1, s_2; t_1, t_2\}$, the following sequence of operations on $\Omega$ is called a \emph{cylinder flip over $\{s_1, s_2; t_1, t_2\}$}.

 \eat{Let $\{s_1, s_2\}$ and $\{t_1, t_2\}$ each be a pair of blocking vertices of $\Omega$.} 

\renewcommand\thesubfigure{(\arabic{subfigure})}
\begin{figure}[H]
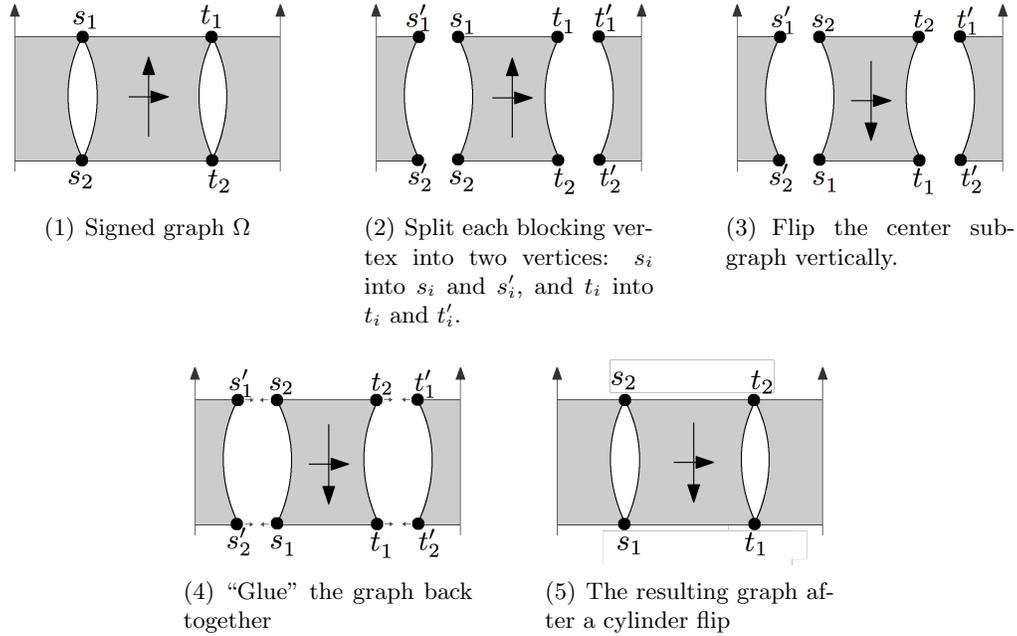

	\centering
    \subfigure[Signed graph $\Omega$]{\includegraphics[scale=0.13]{cylinderflip1-1}}
    \hspace{20pt}
    \subfigure[Split each blocking vertex into two vertices: $s_i$ into $s_i$ and $s_i'$, and $t_i$ into $t_i$ and $t_i'$.]{\label{fig:A1}\includegraphics[scale=0.13]{cylinderflip1-3}}
    \hspace{20pt}
    \subfigure[Flip the center subgraph vertically.]{\includegraphics[scale=0.13]{cylinderflip1-6}}
    \subfigure[``Glue'' the graph back together]{\includegraphics[scale=0.13]{cylinderflip1-8}}
    \hspace{20pt}
    \subfigure[The resulting graph after a cylinder flip]{\includegraphics[scale=0.13]{cylinderflip2}}
    \caption{A cylinder flip over $\{s_1, s_2; t_1, t_2\}$.}
\end{figure}
\end{df}

To show why the cylinder flip preserves a matroid, we first prove a lemma:

\begin{lem}\label{Lem:cylinderflip}Let $H_1$ and $H_2$ be two signed graphs containing only positive edges. Let $s_1, s_2, t_1, t_2 \in V(H_1)$ and $s_1', s_2', t_1', t_2' \in V(H_2)$. Let $\Omega$ (resp. $\Omega'$) be the signed graph obtained by adding two positive edges $\{s_1, s_1'\}$, $\{s_2, s_2'\}$ (resp. $\{s_1, s_2'\}$, $\{s_2, s_1'\}$ ) and two negative edges $\{t_1, t_1'\}$, $\{t_2, t_2'\}$ (resp. $\{t_1, t_2'\}$, $\{t_2, t_1'\}$) between $H_1$ and $H_2$. Then the corresponding matroids $M(\Omega)$ and $M(\Omega')$ of these signed graphs are equivalent.
\renewcommand\thesubfigure{(\alph{subfigure})}
\begin{figure}[H]
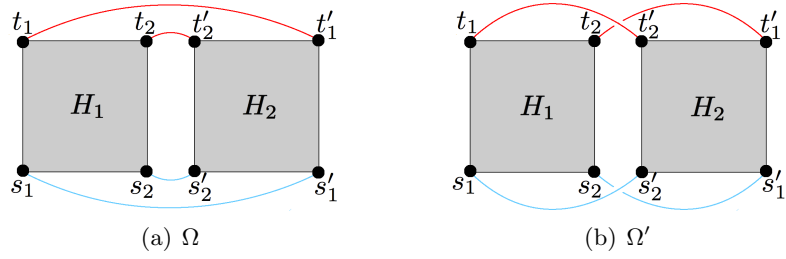

	\centering
	\subfigure[$\Omega$]{\includegraphics[scale=0.13]{flip1}}
	\hspace{20pt}
    \subfigure[$\Omega'$]{\includegraphics[scale=0.13]{flip2}}
\caption{The signed graphs $\Omega$ and $\Omega'$ of Lemma \ref{Lem:cylinderflip}.}
\end{figure}
\end{lem}

\begin{proof}We will prove that if $X \subseteq E(\Omega)$ is a circuit of $\Omega$, then the corresponding set of edges in $\Omega'$ is also a circuit of $\Omega'$. The other direction follows by symmetry.

If $X$ is contained entirely in $H_1$ or $H_2$, then since $\Omega$ and $\Omega'$ have equivalent subgraphs $H_1$ and $H_2$, $X$ is a circuit in $\Omega$ if and only if it is a circuit in $\Omega'$. So we may assume that $X \cap H_1 \neq \emptyset$ and $X \cap H_2 \neq \emptyset$.

If $X \subseteq E(\Omega)$ is a circuit of $\Omega$, then it is either a positive cycle, two negative cycles meeting at a point, or two negative cycles connected by a path.

\underline{Case 1}. Suppose $X$ is a positive cycle, i.e., $X$ contains an even number of negative edges. Since there are only two negative edges, $X$ either contains 0 or 2 negative edges.

If $X$ contains no negative edges, $X$ must be a union of paths $P_1 \cup \{ s_1, s_1' \} \cup P_2 \cup \{ s_2, s_2' \}$ where $P_1$ (resp. $P_2$) is a path contained entirely in $H_1$ (resp. $H_2$). Then the same set of edges in $P_1 \cup \{ s_1, s_2' \} \cup P_2 \cup \{ s_2, s_1' \}$ form a positive cycle in $\Omega'$:

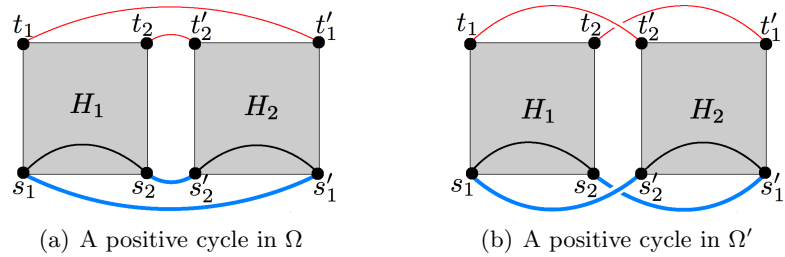
\setcounter{subfigure}{0}
\begin{figure}[H]
	\centering
	\subfigure[A positive cycle in $\Omega$]{\includegraphics[scale=0.13]{flip1_circuit6}}
	\hspace{20pt}
    \subfigure[A positive cycle in $\Omega'$]{\includegraphics[scale=0.13]{flip2_circuit6}}
\end{figure}

If $X$ contains two negative edges, then the following figures enumerate all possible scenarios, up to symmetry. In all scenarios, the same set of edges also correspond to a positive cycle in $\Omega'$:

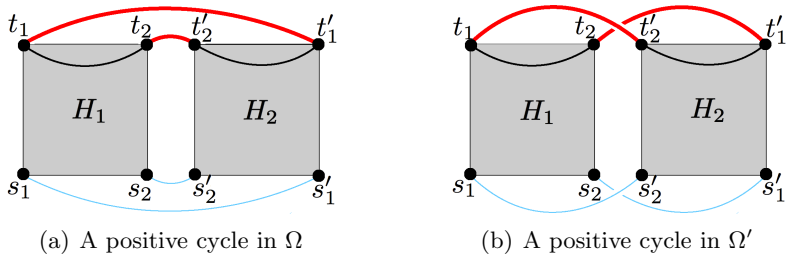
\setcounter{subfigure}{0}
\begin{figure}[H]
	\centering
	\subfigure[A positive cycle in $\Omega$]{\includegraphics[scale=0.13]{flip1_circuit3}}
	\hspace{20pt}
    \subfigure[A positive cycle in $\Omega'$]{\includegraphics[scale=0.13]{flip2_circuit3}}
\end{figure}

\setcounter{subfigure}{0}
\begin{figure}[H]
	\centering
	\subfigure[A positive cycle in $\Omega$]{\includegraphics[scale=0.13]{flip1_circuit4}}
	\hspace{20pt}
    \subfigure[A positive cycle in $\Omega'$]{\includegraphics[scale=0.13]{flip2_circuit4}}
\end{figure}

\setcounter{subfigure}{0}
\begin{figure}[H]
	\centering
	\subfigure[A positive cycle in $\Omega$]{\includegraphics[scale=0.13]{flip1_circuit5}}
	\hspace{20pt}
    \subfigure[A positive cycle in $\Omega'$]{\includegraphics[scale=0.13]{flip2_circuit5}}
\end{figure}

\underline{Case 2}. Suppose $X$ contains two negative cycles in $\Omega$. Since $\Omega$ contains only two negative edges $\{t_2, t_2'\}$ and $\{t_1, t_1'\}$, each negative cycle must contain one of these negative edges. Moreover, since there are only two other edges $\{s_1, s_1'\}$ and $\{s_2, s_2'\}$ connecting $H_1$ and $H_2$, each negative cycle must use one of these edges in order to close the walk.

Suppose the two negative cycles meet at a vertex $v$ in $\Omega$. Without loss of generality, we may assume $v \in H_1$. Then the following figures enumerate all possible scenarios, up to symmetry. In all scenarios, the same set of edges also correspond to two negative cycles meeting at a point in $\Omega'$:
\setcounter{subfigure}{0}
\begin{figure}[H]
	\centering
	\subfigure[Two negative cycles meeting in a point in $\Omega$]{\includegraphics[scale=0.13]{flip1_circuit1}}
	\hspace{20pt}
    \subfigure[Two negative cycles meeting in a point in $\Omega'$]{\includegraphics[scale=0.13]{flip2_circuit1}}
\end{figure}

\setcounter{subfigure}{0}
\begin{figure}[H]
	\centering
	\subfigure[Two negative cycles meeting in a point in $\Omega$]{\includegraphics[scale=0.13]{flip1_circuit0}}
	\hspace{20pt}
    \subfigure[Two negative cycles meeting in a point in $\Omega'$]{\includegraphics[scale=0.13]{flip2_circuit0}}
\end{figure}

Otherwise, suppose the two negative cycles are joined by a path $P$ in $\Omega$. Without loss of generality, we may assume $P \subseteq H_1$. Then the same set of edges also correspond to two negative cycles joined by a path in $\Omega'$:

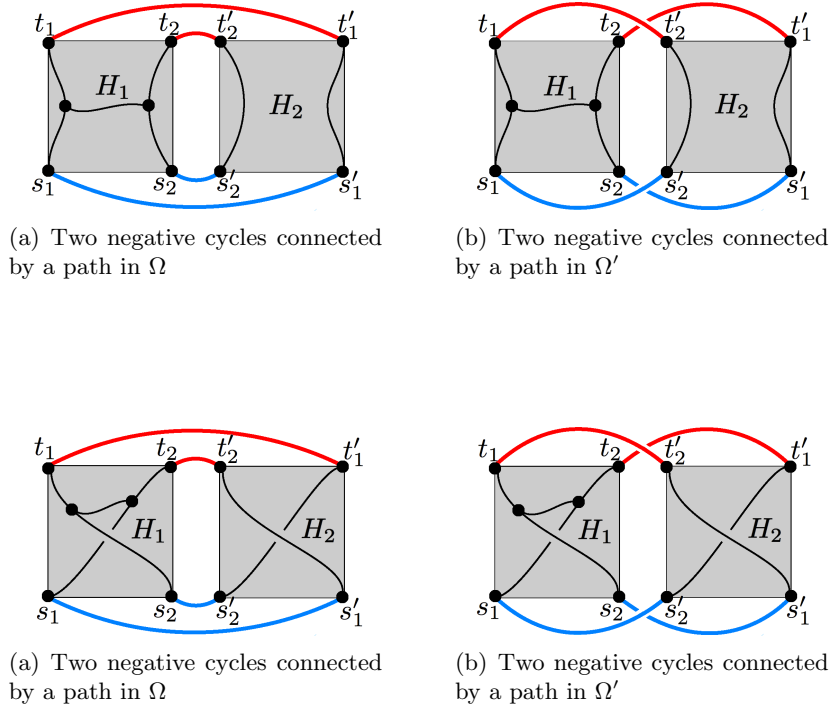
\setcounter{subfigure}{0}
\begin{figure}[H]
	\centering
	\subfigure[Two negative cycles connected by a path in $\Omega$]{\includegraphics[scale=0.13]{flip1_circuit2}}
	\hspace{20pt}
    \subfigure[Two negative cycles connected by a path in $\Omega'$]{\includegraphics[scale=0.13]{flip2_circuit2}}
\end{figure}

\setcounter{subfigure}{0}
\begin{figure}[H]
	\centering
	\subfigure[Two negative cycles connected by a path in $\Omega$]{\includegraphics[scale=0.13]{flip1_circuit-1}}
	\hspace{20pt}
    \subfigure[Two negative cycles connected by a path in $\Omega'$]{\includegraphics[scale=0.13]{flip2_circuit-1}}
\end{figure}

We have shown that any circuit $X \subseteq E(\Omega)$ of $\Omega$ is also a circuit in $\Omega'$. The proof of the other direction follows by symmetry. Therefore, we conclude that $M(\Omega) = M(\Omega')$.
\end{proof}

\begin{thm}[Cylinder flip]\label{Thm:cylinderflip2}
Let $\Omega$ be a cylinder graph on the blocking pair vertices $\{s_1, s_2; t_1, t_2\}$. If $\Omega'$ is obtained from $\Omega$ via a cylinder flip over $\{s_1, s_2; t_1, t_2\}$, then $M(\Omega) = M(\Omega')$.

\renewcommand\thesubfigure{(\alph{subfigure})}
\begin{figure}[H]
	\centering
	\subfigure[$\Omega$]{\includegraphics[scale=0.13]{cylinderflip1-1}}
	\hspace{20pt}
    \subfigure[$\Omega'$]{\includegraphics[scale=0.13]{cylinderflip2}}
\caption{The signed graphs $\Omega$ and $\Omega'$ of Cor. \ref{Thm:cylinderflip2}.}
\end{figure}
\end{thm}

\begin{proof} By Lemma \ref{Lem:cylinderflip}, the matroids of the following signed graphs $H$ and $H'$ are equivalent:

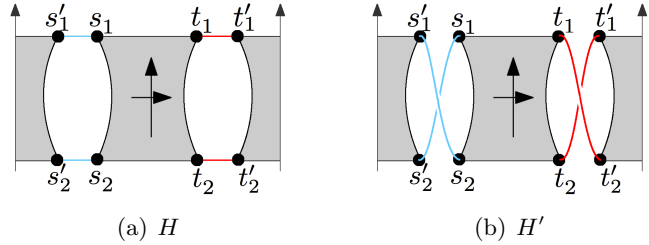
\renewcommand\thesubfigure{(\alph{subfigure})}
\begin{figure}[H]
	\centering
	\subfigure[$H$]{\includegraphics[scale=0.13]{cylinderflip1_1}}
	\hspace{20pt}
    \subfigure[$H'$]{\includegraphics[scale=0.13]{cylinderflip2_1}}
\end{figure}
Resigning all edges incident with $t_1$ and $t_2$, then contracting the edges $\{s_1, s_1'\}$, $\{s_2, s_2' \}$, $\{t_1, t_1'\}$, and $\{t_2, t_2'\}$ in $H$ gives $\Omega$. Similarly, resigning all edges incident with $t_1$ and $t_2$, then contracting the edges $\{s_1, s_2'\}$, $\{s_2, s_1' \}$, $\{t_1, t_2'\}$, and $\{t_2, t_1'\}$ in $H'$ gives $\Omega'$. Since contraction in the signed graph corresponds to the abstract matroid contraction, we conclude that $M(\Omega) = M(\Omega')$.
\end{proof}

\bigskip
\subsubsection{Degenerate cases}\label{subsubsection:degenerate_case}
With a little more work, we can also prove Theorem \ref{Thm:cylinderflip2} for the degenerate cases, e.g. when $s_1 = t_1$.

\renewcommand\thesubfigure{(\arabic{subfigure})}
\begin{figure}[H]
	\centering
    \subfigure[Signed graph $\Omega$]{\includegraphics[scale=0.13]{degenerate_cylinderflip1-1}}
    \hspace{20pt}
    \subfigure[Split each blocking vertex into two vertices: $s_i$ into $s_i$ and $s_i'$, and $t_i$ into $t_i$ and $t_i'$.]{\label{fig:A1}\includegraphics[scale=0.13]{degenerate_cylinderflip1-2}}
    \hspace{20pt}
    \subfigure[Flip the center subgraph vertically.]{\includegraphics[scale=0.13]{degenerate_cylinderflip1-3}}
    
    \subfigure[``Glue'' the graph back together.]{\includegraphics[scale=0.13]{degenerate_cylinderflip1-4}}
    \hspace{20pt}
    \subfigure[The resulting graph after a cylinder flip.]{\includegraphics[scale=0.13]{degenerate_cylinderflip2}}
    \caption{A cylinder flip in the degenerate case when $s_1=t_1$.}
\end{figure}

\subsection{Example}\label{subsection:flip_example}

Consider the dyadic matroid $M$ of rank 6 on 10 elements whose $\mathbb{D}$-representation is
\begin{align*}
\begin{footnotesize}
	\left(\begin{array}{rrrrrrrrrr}
	1 & 0 & 0 & 0 & 0 & 2 & 0 & 0 & 1 &	0 \\
	0 & 1 & 0 & 0 & 0 & 1 & 0 & 2 & 0 &	0 \\
	0 & 0 & 1 & 0 & 0 & 0 & 1 & 1 & 1 &	0 \\
	0 & 0 & 0 & 1 & 0 & 0 & 1 & 1 & 0 &	0 \\
	0 & 0 & 0 & 0 & 1 & 1 & 0 & 2 & 2 &	0 \\
	0 & 0 & 0 & 0 & 0 & 2 & 2 & 0 & 0 &	1
	\end{array}\right).
\end{footnotesize}
\end{align*}

Two of its signed-graphic representations are given below\footnote{A complete list of its signed-graphic representations can be found in Appendix \ref{App:sameAll}.}:
\renewcommand\thesubfigure{(\arabic{subfigure})}
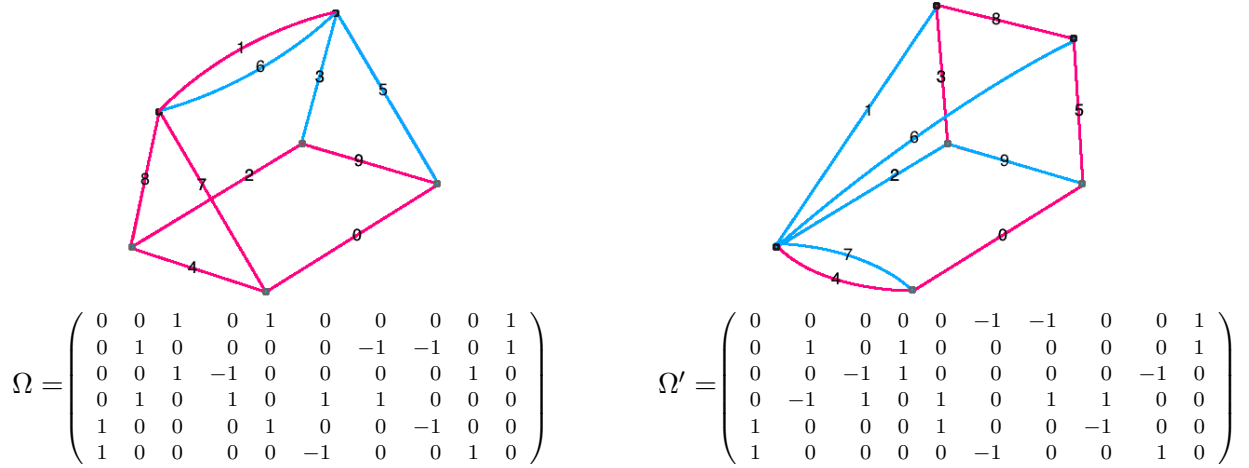
\begin{figure}[H]
\minipage[r]{0.48\textwidth}
  \center
  \hspace*{20mm}\includegraphics[scale=0.35]{A0}
  \vspace{-5mm}
  \caption*{\label{fig:sameAll1}$\Omega =$\usebox{\sameAllone}}
\endminipage\hfill
\minipage[r]{0.48\textwidth}
  \center
  \hspace*{20mm}\includegraphics[scale=0.35]{A8}
  \vspace{-5mm}
  \caption*{\label{fig:sameAll2}$\Omega' =$\usebox{\sameAlltwo}}
\endminipage\hfill
\end{figure}

The cylinder flip can be used to get from $\Omega$ to $\Omega'$, as illustrated below. 

\begin{figure}[H]
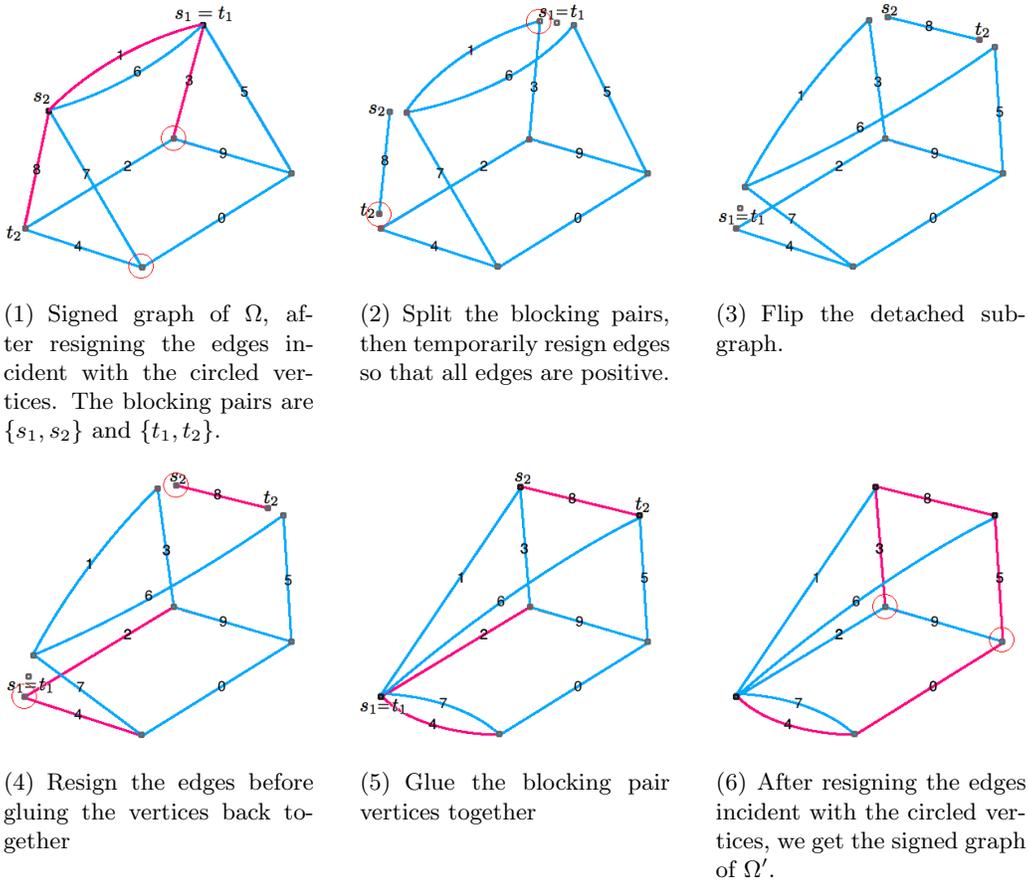

	\centering
    \subfigure[Signed graph of $\Omega$, after resigning the edges incident with the circled vertices. The blocking pairs are $\{s_1, s_2 \}$ and $\{t_1, t_2\}$.]{\label{fig:A1}\includegraphics[width=.25\textwidth]{A2}}
	\hspace{10pt}
	\subfigure[Split the blocking pairs, then temporarily resign edges so that all edges are positive.]{\label{fig:A3}\includegraphics[width=.25\textwidth]{A3}}
	\hspace{10pt}
	\subfigure[Flip the detached subgraph.]{\label{fig:A3}\includegraphics[width=.25\textwidth]{A4}}
	\hspace{10pt}
	\subfigure[Resign the edges before gluing the vertices back together]{\label{fig:A5}\includegraphics[width=.25\textwidth]{A5}}
	\hspace{10pt}
	\subfigure[Glue the blocking pair vertices together]{\label{fig:A6}\includegraphics[width=.25\textwidth]{A6}}
	\hspace{10pt}
	\subfigure[After resigning the edges incident with the circled vertices, we get the signed graph of $\Omega'$.]{\label{fig:A7}\includegraphics[width=.25\textwidth]{A7}}
	\caption{A sequence of matroid-preserving operations to get from $\Omega$ to $\Omega'$.}
\end{figure}

\section{Restricting column scaling factors to $\pm 1$}\label{Section:Restricting_column_scaling}

If $M$ is a rank-$r$ matroid with $r \geq 1$, then every matrix that represents $M$ over a field $\mathbb{F}$ is projectively equivalent to a standard representative matrix $[I_r \mid D]$ for $M$. \eat{Suppose we are given a representation $[I_r \mid D_1]$ of a matroid over a field $\mathbb{F}$.} If we allow row and column scalings, then Theorem \ref{brylawski} (\cite[6.4.7]{oxley}) shows that we can transform certain entries of $D_1$ into any nonzero elements $\theta_1, \theta_2, \ldots, \theta_k \in \mathbb{F}$ that we want. On the other hand, if we restrict column scaling factors to $\pm 1$, then Lemma \ref{notbrylawski} shows that our choices of $\theta_i$'s become very limited. Before stating Theorem \ref{brylawski}, we first require some preliminaries.

\subsection{Preliminaries}

Given a standard representation $[I_r \mid D]$ for a matroid $M$, let its columns be labeled, in order, $e_1, e_2, \ldots, e_n$. If $B$ is the basis $\{e_1, e_2, \ldots, e_r\}$, then it is natural to label the rows of $[I_r \mid D]$ by $e_1, e_2, \ldots, e_r$. Hence $D$ has its rows labeled by $e_1, e_2, \ldots, e_r$ and its columns labeled by $e_{r+1}, e_{r+2}, \ldots, e_n$.

\eat{For all $k \in \{ r+1, r+2, \ldots, n \}$, the \emph{fundamental circuit} is 
	\[ C(e_k, B) := e_k \cup \{ e_i: e_i \in B, \; D_{e_i, e_k} \neq 0 \}. \]}
The \emph{$B$-fundamental-circuit incidence matrix} of $M$, which we denote by $D^{\#}$, is the matrix obtained from $D$ by replacing each nonzero entry of $D$ by a 1, where the rows and columns of $D^\#$ inherit their labels from $D$. \eat{, so that the columns of $D^{\#}$ are precisely the incidence vectors of the sets $C(e_k, B) - e_k$ where the rows and columns of $D^\#$ inherit their labels from $D$.} $[I_r \mid D^\#]$ is sometimes called a \emph{partial representation} for $M$ because of the following result (\cite[6.4.1]{oxley}):

\begin{prop}\label{6.4.1}
Let $[I_r \mid D_1]$ and $[I_r \mid D_2]$ be matrices over the fields $\mathbb{K}_1$ and $\mathbb{K}_2$ with the columns of each labeled, in order, by $e_1, e_2, \ldots, e_n$. If the identity map on $\{ e_1, e_2, \ldots, e_n \}$ is an isomorphism from $M[I_r \mid D_1]$ to $M[I_r \mid D_2]$, then $D_1^\# = D_2^\#$.
\end{prop}


\eat{Let $M$ be a rank-$r$ matroid and $B=\seq*{e_1, e_2, \ldots, e_r}$ be a basis for $M$. Let $X$ be the $B$-fundamental-circuit incidence matrix of $M$, the columns of this matrix being labeled by $e_{r+1}, e_{r+2}, \ldots, e_n$.

By Prop. \ref{6.4.1}, if $[I_r \mid D]$ is an $\mathbb{F}$-representation of $M$ with columns labeled, in order, $e_1, e_2, \ldots, e_n$, then $[I_r \mid D^\#] = [I_r \mid X]$. Hence, the task of finding an $\mathbb{F}$-representation for $M$ is equivalent to finding the specific elements of $\mathbb{F}$ that correspond to the nonzero elements of $D^\#$.

\begin{itemize}
\item By column scaling, WMA the first nonzero entry in each column of $D$ is 1.
\item By row scaling, WMA the first nonzero entry in each row of $D$ is 1.
\end{itemize}

The rows of $D^\#$ are indexed by $e_1, e_2, \ldots, e_r$ and its columns by $e_{r+1}, e_{r+2}, \ldots, e_n$.}

Let $G(D^\#)$ denote the \emph{associated} simple \emph{bipartite graph}, that is, $G(D^\#)$ has vertex classes $\seq*{e_1, e_2, \ldots, e_r}$ and $\seq*{e_{r+1}, e_{r+2}, \ldots, e_n}$, and two vertices $e_i$ and $e_j$ are adjacent if and only if $D^\#_{e_i,e_j}=1$. Equivalently, $G(D^\#)$ is the \emph{fundamental graph} $G_B(M)$ associated with the basis $B=\seq*{e_1, e_2, \ldots, e_r}$ of $M$ (See \cite[4.3.2]{oxley}).


\subsection{Brylawski-Lucas Theorem}

The following theorem gives a best-possible lower bound on the number of entries of $D$ whose values can be predetermined by row and column scaling, and also provides an easy way to see whether or not two representations are projectively equivalent.

\bigskip

\begin{thm}[Brylawski and Lucas 1976]\label{brylawski}
For a field $\mathbb{F}$, let the $r \times n$ matrix $[I_r \mid D_1]$ be an $\mathbb{F}$-representation of a matroid $M$. Let $\seq*{b_1, b_2, \ldots, b_k}$ be a basis of the cycle matroid of $G(D_1^\#)$. Then 
\begin{enumerate}
\item $k = n - \omega(G(D_1^\#))$.
\item If $(\theta_1, \theta_2, \ldots, \theta_k)$ is an ordered $k$-tuple of nonzero elements of $\mathbb{F}$, then $M$ has a unique $\mathbb{F}$-representation $[I_r \mid D_2]$ that is projectively equivalent to $[I_r \mid D_1]$ such that, for each $i \in \seq*{1, 2, \ldots, k}$, the entry of $D_2$ corresponding to $b_i$ is $\theta_i$. Indeed, $[I_r \mid D_2]$ can be obtained from $[I_r \mid D_1]$ by a sequence of row and column scalings.
\end{enumerate}
\end{thm}

We repeat the proof of Theorem \ref{brylawski} from \cite[6.4.7]{oxley} here, in order to contrast with that of Lemma \ref{notbrylawski}. We omit the proof that $[I_r \mid D_2]$ is unique, and refer the reader to \cite[6.4.7]{oxley}.

\begin{proof}
(1) follows from the fact that, if $M=M(G)$ where $G$ is a graph, then $r(M) = |V(G)| - \omega(G)$.

To prove (2), let $G_1 \subseteq G(D_1^\#)$ be a forest, and for each edge $x \in E(G_1)$, $\theta(x)$ is a nonzero element of $\mathbb{F}$. We will show by induction on $|E(G_1)|$ that, by a sequence of row and column scalings of $[I_r \mid D_1]$, we can obtain a matrix $[I_r \mid D_2]$ such that, for all edges $x \in E(G_1)$, the entry in $D_2$ corresponding to $x$ is $\theta(x)$.

The hypothesis is trivially true for $|E(G_1)| = 0$. Assume it is true for $|E(G_1)| < m$ and let $|E(G_1)| = m \geq 1$. As $G_1$ is a forest with at least one edge, it has a vertex $v$ of degree one. Let $y \in E(G_1)$ be the edge incident with $v$. By induction, there is a sequence of row and column scalings of $[I_r \mid D_1]$ to obtain $[I_r \mid D_2]$ such that, for all edges $x \in E(G_1 \backslash y)$, the entry of $D_2$ corresponding to $x$ is $\theta(x)$.

Now $v$ corresponds to a row or a column of $D_2$. If $v$ corresponds to a column, then we can scale that column to make the $y$-entry equal to $\theta(y)$ without altering any of the $(G_1\backslash y)$-entries. If $v$ corresponds to a row, then as $y$ is the only edge meeting $v$, none of the entries of $D_2$ corresponding to edges of $G_1 \backslash y$ is in this row. So we can multiply this row by an appropriate nonzero scalar $t$ to make the $y$-entry equal to $\theta(y)$ without changing any of the $(G_1 \backslash y)$-entries. This multiplication may alter the entry in row $v$ of $I_r$, but we can fix this by multiplying the corresponding column of $I_r$ by $t^{-1}$ without affecting any other entries.

We conclude by induction that there is a sequence of row and column scalings of $[I_r \mid D_1]$ to obtain $[I_r \mid D_2]$ in the desired form.
\end{proof}

\subsubsection{Example} Let $B = \{e_1,e_2,e_3\}$. For the non-Fano $F_7^-$, the $B$-fundamental-circuit incidence matrix over a field $\mathbb{F}$ of characteristic $\neq2$ is given below:

\renewcommand\thesubfigure{(\alph{subfigure})}
\begin{figure}[H]
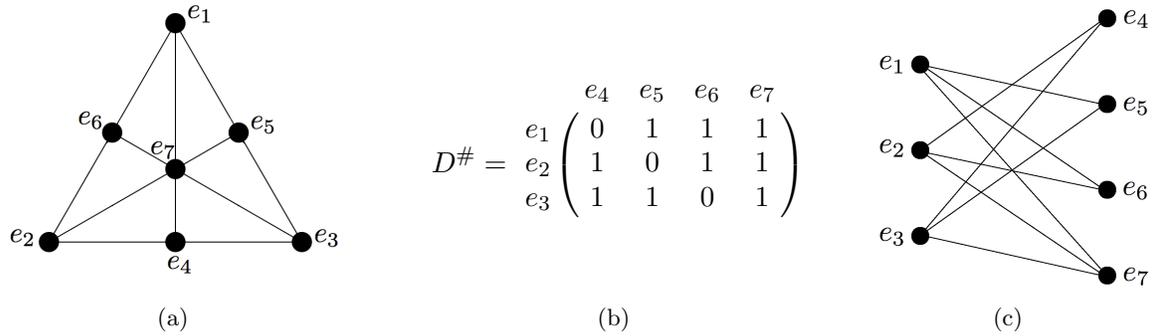

    \centering
    \subfigure[]{\includegraphics[scale=0.19]{nonfano}}
    \hspace{20pt}
    \subfigure[]{\raisebox{45pt}{$D^\# = \bordermatrix{ & e_4 & e_5 & e_6 & e_7\cr
                e_1 & 0 & 1 & 1 & 1\cr
                e_2 & 1 & 0 & 1 & 1\cr
                e_3 & 1 & 1 & 0 & 1}$}}
    \hspace{20pt}
    \subfigure[]{\includegraphics[scale=0.117]{bipartite}}
\caption{(a) A geometric representation for $F_7$, (b) the $B$-fundamental-circuit incidence matrix, and (c) its associated bipartite graph.}
\end{figure}

The boxed entries in $D^\#$ below correspond to a basis of the cycle matroid of $G(D^\#)$.
\[D^\# = \bordermatrix{ & e_4 & e_5 & e_6 & e_7\cr
                e_1 & 0 & \fbox{1} & \fbox{1} & \fbox{1}\cr
                e_2 & \fbox{1} & 0 & 1 & \fbox{1}\cr
                e_3 & 1 & \fbox{1} & 0 & 1}\]

By Theorem \ref{brylawski}, we can independently assign arbitrary nonzero elements of $\mathbb{F}$ to each of those boxed entries. Note that $n - \omega(G(D^\#)) = 7 - 1 = 6$.

\subsection{Restricting column scaling}

If we restrict column scaling factors to $\pm 1$, then we cannot choose $\theta(x)$ freely as we did in Theorem $\ref{brylawski}$, and instead we are only restricted to $\theta(x)$ such that the entry corresponding to $x$ is $\pm \theta(x)$. The following lemma illustrates this.

\eat{\begin{lem}For a field $\mathbb{F}$, let the $r \times n$ matrix $[I_r \mid D_1]$ be an $\mathbb{F}$-representation of a matroid $M$. Let $\seq*{b_1, b_2, \ldots, b_k}$ be a basis of the cycle matroid of $G(D_1^\#)$, and $(\theta_1, \theta_2, \ldots, \theta_k)$ an ordered $k$-tuple of nonzero elements of $\mathbb{F}$ such that the entry of $D_1$ corresponding to $b_i$ is $\pm \theta_i$. Then $M$ has a unique $\mathbb{F}$-representation $[I_r \mid D_2]$ that is projectively equivalent to $[I_r \mid D_1]$ such that, for each $i \in \seq*{1, 2, \ldots, k}$, the entry of $D_2$ corresponding to $b_i$ is $\theta_i$. Indeed, $[I_r \mid D_2]$ can be obtained from $[I_r \mid D_1]$ by a sequence of row and column scalings.
\end{lem}

In Theorem \ref{brylawski}, if we restrict the column scaling factors to $\pm 1$, then we can only choose $\theta_i$ such that the entry of $D_1$ corresponding to $b_i$ is $\pm \theta_i$.
\end{lem}} 

\begin{lem}\label{notbrylawski} With notation as in Theorem \ref{brylawski}, let $(\phi_1, \phi_2, \ldots, \phi_k)$ be an ordered $k$-tuple such that the entry of $D_1$ corresponding to $b_i$ is $\pm\phi_i$. Let $[I_r \mid D']$ be the matrix such that, for each $i \in \{ 1, 2, \ldots, k\}$, the entry of $D'$ corresponding to $b_i$ is $\phi_i$. Then we can obtain $[I_r \mid D']$ from $[I_r \mid D_1]$ by a sequence of row and column scalings by $\pm 1$.
\end{lem}

\begin{proof}
Let $G_1, D_2, v$, and $y$ be as in the proof of Theorem \ref{brylawski}. Again, $v$ can label a row or a column of $D_2$.

If $v$ corresponds to a column, then we can multiply the column by $\pm 1$ to make the $y$-entry equal to $\phi(y)$ without altering any of the $(G_1 \backslash y)$-entries. If $v$ corresponds to a row, then we can multiply this row by an appropriate nonzero scalar $t$ to make the $y$-entry equal to $\phi(y)$ without changing any of the $(G_1 \backslash y)$-entries. However, this multiplication may alter the entry in row $v$ of $I_r$. Since we only allow column scaling by $-1$, we can only fix this (by multiplying the corresponding column of $I_r$ by $\pm 1$) if $t = \pm 1$.

We note that in either case, $\phi(y)$ can only be $\pm 1$ times the original entry of $D_1$ corresponding to $y$, because column scaling is restricted to scaling by $\pm 1$.
\end{proof}

\section{Conclusion}

We have shown that whether to allow column scaling in relating two signed-graphic representations is a relevant question. Our examples show that column scaling is sometimes necessary in order to transform one signed-graphic representation into another; moreover, there exist many collections of signed-graphic representations that row-reduce to the same standard form. 

Given a binet matrix which can be obtained by row-reducing several different signed graphs, Musitelli does not mention which of the signed graphs is found by his algorithm, and his algorithm does not find all the signed graphs. It is possible that his algorithm makes some consistent canonical choice, or finds certain representations more attractive than others, but this is not mentioned in his paper. Because Musitelli's paper ignores the existence of both other row-equivalent and other row-inequivalent representations, there is some doubt as to what his algorithm does.

\newpage\appendix
\section{Examples of signed-graphic representations}\label{App:SignedGraphs}

\subsection{A matroid with some row-equivalent, some non-row-equivalent signed-graphic representations}

The following dyadic matroid of rank 6 on 10 elements whose $\mathbb{D}$-representation is
	\begin{align*}
	\begin{footnotesize}
	\left(\begin{array}{rrrrrrrrrr}
	1 & 0 & 0 & 0 & 0 & 0 & 2 & 1 & 2 &	0 \\
	0 & 1 & 0 & 0 & 0 & 1 & 0 & 1 & 1 &	0 \\
	0 & 0 & 1 & 0 & 0 & 1 & 0 & 0 & 1 &	0 \\
	0 & 0 & 0 & 1 & 0 & 0 & 1 & 2 & 2 &	0 \\
	0 & 0 & 0 & 0 & 1 & 0 & 1 & 0 & 2 &	0 \\
	0 & 0 & 0 & 0 & 0 & 0 & 0 & 2 & 2 &	1
	\end{array}\right)
	\end{footnotesize}\end{align*}

has 11 signed-graphic representations, some of which are row-equivalent and some of which are not.

The following two signed-graphic representations are not row-equivalent to any of the other representations; in other words, column scaling is required in order to transform either representation into a different signed-graphic representation:
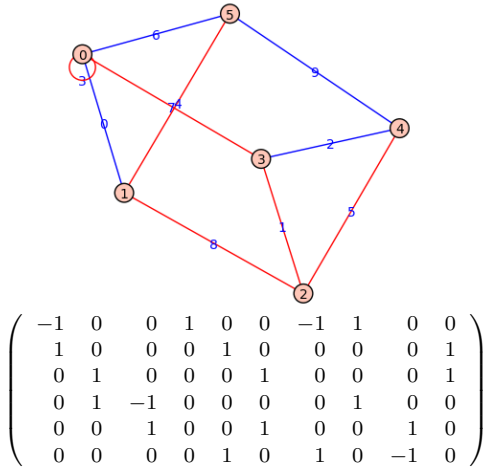
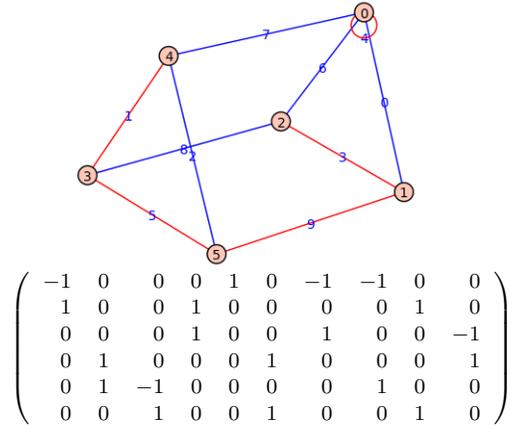
\begin{figure}[H]
\minipage[t]{0.48\textwidth}
  \center
  \includegraphics[scale=0.5]{diffSome10}
  \vspace{-5mm}
  \caption*{\label{fig:diffSome10}\usebox{\diffSometen}}
\endminipage\hfill
\minipage[t]{0.48\textwidth}
  \center
  \hspace*{10mm}\includegraphics[scale=0.5]{diffSome11}
  \vspace{-5mm}
  \caption*{\label{fig:diffSome11}\usebox{\diffSomeeleven}}
\endminipage\hfill
\end{figure}

The remaining representations are all row-equivalent to one another:
\begin{figure}[H]
\minipage[t]{0.48\textwidth}
  \center
  \hspace*{15mm}\includegraphics[scale=0.5]{diffSome1}
  \vspace{-5mm}
  \caption*{\label{fig:diffSome1}\usebox{\diffSomeone}}
\endminipage\hfill
\minipage[t]{0.48\textwidth}
  \center
  \hspace*{15mm}\includegraphics[scale=0.5]{diffSome2}
  \vspace{-5mm}
  \caption*{\label{fig:diffSome2}\usebox{\diffSometwo}}
\endminipage\hfill
\end{figure}

\begin{figure}[H]
\minipage[t]{0.48\textwidth}
  \center
  \hspace*{20mm}\includegraphics[scale=0.6]{diffSome3}
  \vspace{-5mm}
  \caption*{\label{fig:diffSome3}\usebox{\diffSomethree}}
\endminipage\hfill
\minipage[t]{0.48\textwidth}
  \center
  \hspace*{15mm}\includegraphics[scale=0.5]{diffSome4}
  \vspace{-5mm}
  \caption*{\label{fig:diffSome4}\usebox{\diffSomefour}}
\endminipage\hfill
\end{figure}

\begin{figure}[H]
\minipage[t]{0.48\textwidth}
  \center
  \hspace*{20mm}\includegraphics[scale=0.6]{diffSome5}
  \vspace{-5mm}
  \caption*{\label{fig:diffSome5}\usebox{\diffSomefive}}
\endminipage\hfill
\minipage[t]{0.48\textwidth}
  \center
  \hspace*{15mm}\includegraphics[scale=0.5]{diffSome6}
  \vspace{-5mm}
  \caption*{\label{fig:diffSome6}\usebox{\diffSomesix}}
\endminipage\hfill
\end{figure}

\begin{figure}[H]
\minipage[t]{0.48\textwidth}
  \center
  \hspace*{15mm}\includegraphics[scale=0.5]{diffSome7}
  \vspace{-5mm}
  \caption*{\label{fig:diffSome7}\usebox{\diffSomeseven}}
\endminipage\hfill
\minipage[t]{0.48\textwidth}
  \center
  \hspace*{15mm}\includegraphics[scale=0.5]{diffSome8}
  \vspace{-5mm}
  \caption*{\label{fig:diffSome8}\usebox{\diffSomeeight}}
\endminipage\hfill
\end{figure}

\begin{figure}[H]
\minipage[t]{0.48\textwidth}
  \center
  \hspace*{15mm}\includegraphics[scale=0.5]{sameAll9}
  \vspace{-5mm}
  \caption*{\label{fig:diffSome9}\usebox{\diffSomenine}}
\endminipage\hfill
\end{figure}

\newpage
\subsection{A matroid with all row-equivalent signed-graphic representations\label{App:sameAll}}

The following dyadic matroid of rank 6 on 10 elements whose $\mathbb{D}$-representation is
	\begin{align*}
	\begin{footnotesize}
	\left(\begin{array}{rrrrrrrrrr}
	1 & 0 & 0 & 0 & 0 & 2 & 0 & 0 & 1 &	0 \\
	0 & 1 & 0 & 0 & 0 & 1 & 0 & 2 & 0 &	0 \\
	0 & 0 & 1 & 0 & 0 & 0 & 1 & 1 & 1 &	0 \\
	0 & 0 & 0 & 1 & 0 & 0 & 1 & 1 & 0 &	0 \\
	0 & 0 & 0 & 0 & 1 & 1 & 0 & 2 & 2 &	0 \\
	0 & 0 & 0 & 0 & 0 & 2 & 2 & 0 & 0 &	1
	\end{array}\right)
	\end{footnotesize}
	\end{align*}
has 15 signed-graphic representations, every pair of which is row-equivalent:

\begin{figure}[H]
\minipage[t]{0.48\textwidth}
  \center
  \hspace*{5mm}\includegraphics[scale=0.5]{sameAll1}
  \vspace{-5mm}
  \caption*{\label{fig:sameAll1}\usebox{\sameAllone}}
\endminipage\hfill
\minipage[t]{0.48\textwidth}
  \center
  \hspace*{15mm}\includegraphics[scale=0.5]{sameAll2}
  \vspace{-5mm}
  \caption*{\label{fig:sameAll2}\usebox{\sameAlltwo}}
\endminipage\hfill
\end{figure}

\begin{figure}[H]
\minipage[t]{0.48\textwidth}
  \center
  \hspace*{15mm}\includegraphics[scale=0.5]{sameAll3}
  \vspace{-5mm}
  \caption*{\label{fig:sameAll3}\usebox{\sameAllthree}}
\endminipage\hfill
\minipage[t]{0.48\textwidth}
  \center
  \hspace*{5mm}\includegraphics[scale=0.5]{sameAll4}
  \vspace{-5mm}
  \caption*{\label{fig:sameAll4}\usebox{\sameAllfour}}
\endminipage\hfill
\end{figure}

\begin{figure}[H]
\minipage[t]{0.48\textwidth}
  \center
  \hspace*{15mm}\includegraphics[scale=0.6]{sameAll5}
  \vspace{-5mm}
  \caption*{\label{fig:sameAll5}\usebox{\sameAllfive}}
\endminipage\hfill
\minipage[t]{0.48\textwidth}
  \center
  \hspace*{10mm}\includegraphics[scale=0.5]{sameAll6}
  \vspace{-5mm}
  \caption*{\label{fig:sameAll6}\usebox{\sameAllsix}}
\endminipage\hfill
\end{figure}

\begin{figure}[H]
\minipage[t]{0.48\textwidth}
  \center
  \hspace*{15mm}\includegraphics[scale=0.6]{sameAll7}
  \vspace{-5mm}
  \caption*{\label{fig:sameAll7}\usebox{\sameAllseven}}
\endminipage\hfill
\minipage[t]{0.48\textwidth}
  \center
  \hspace*{15mm}\includegraphics[scale=0.5]{sameAll8}
  \vspace{-5mm}
  \caption*{\label{fig:sameAll8}\usebox{\sameAlleight}}
\endminipage\hfill
\end{figure}

\begin{figure}[H]
\minipage[t]{0.48\textwidth}
  \center
  \hspace*{5mm}\includegraphics[scale=0.5]{sameAll9}
  \vspace{-5mm}
  \caption*{\label{fig:sameAll9}\usebox{\sameAllnine}}
\endminipage\hfill
\minipage[t]{0.48\textwidth}
  \center
  \hspace*{15mm}\includegraphics[scale=0.5]{sameAll10}
  \vspace{-5mm}
  \caption*{\label{fig:sameAll10}\usebox{\sameAllten}}
\endminipage\hfill
\end{figure}

\begin{figure}[H]
\minipage[t]{0.48\textwidth}
  \center
  \hspace*{15mm}\includegraphics[scale=0.5]{sameAll11}
  \vspace{-5mm}
  \caption*{\label{fig:sameAll11}\usebox{\sameAlleleven}}
\endminipage\hfill
\minipage[t]{0.48\textwidth}
  \center
  \hspace*{5mm}\includegraphics[scale=0.5]{sameAll12}
  \vspace{-5mm}
  \caption*{\label{fig:sameAll12}\usebox{\sameAlltwelve}}
\endminipage\hfill
\end{figure}

\begin{figure}[H]
\minipage[t]{0.48\textwidth}
  \center
  \hspace*{10mm}\includegraphics[scale=0.5]{sameAll13}
  \vspace{-5mm}
  \caption*{\label{fig:sameAll13}\usebox{\sameAllthirteen}}
\endminipage\hfill
\minipage[t]{0.48\textwidth}
  \center
  \hspace*{15mm}\includegraphics[scale=0.5]{sameAll14}
  \vspace{-5mm}
  \caption*{\label{fig:sameAll14}\usebox{\sameAllfourteen}}
\endminipage\hfill
\end{figure}

\begin{figure}[H]
\minipage[t]{0.48\textwidth}
  \center
  \hspace*{15mm}\includegraphics[scale=0.5]{sameAll15}
  \vspace{-5mm}
  \caption*{\label{fig:sameAll15}\usebox{\sameAllfifteen}}
\endminipage\hfill
\end{figure}

\newpage
\subsection{A matroid with all non-row-equivalent signed-graphic representations}

The following dyadic matroid of rank 5 on 9 elements whose $\mathbb{D}$-representation is

	\begin{align*}
	\begin{footnotesize}
	\left(\begin{array}{rrrrrrrrr}
	1 & 0 & 0 & 0 & 2 & 0 & 1 & 1 & 0 \\
	0 & 1 & 0 & 0 & 1 & 1 & 1 & 1 & 0 \\
	0 & 0 & 1 & 0 & 1 & 1 & 1 & 0 & 0 \\
	0 & 0 & 0 & 1 & 2 & 0 & 0 & 2 & 0 \\
	0 & 0 & 0 & 0 & 1 & 2 & 2 & 2 & 1
	\end{array}\right)
	\end{footnotesize}
	\end{align*}

has 3 signed-graphic representations, every pair of which is not row-equivalent:

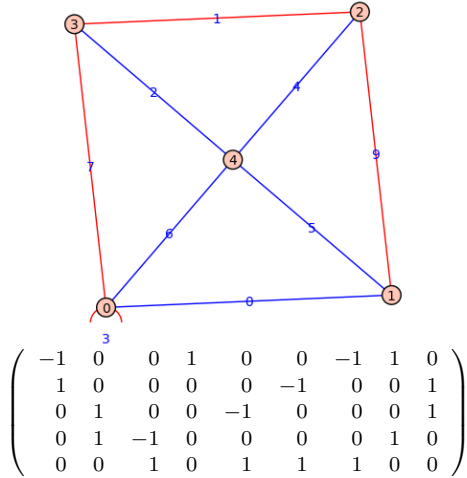
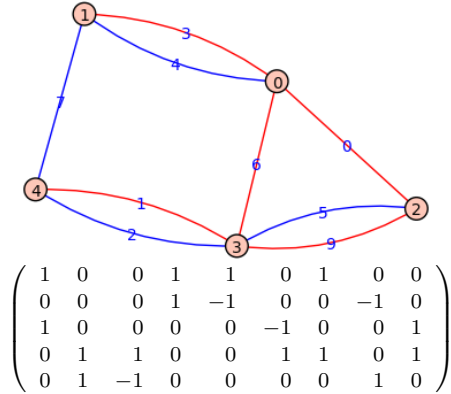
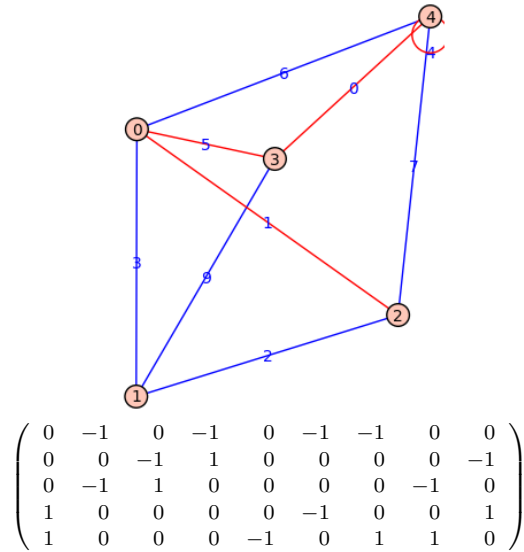
\begin{figure}[H]
\minipage[t]{0.48\textwidth}
  \center
  \includegraphics[scale=0.5]{diffAll1}
  \vspace{-5mm}
  \caption*{\label{fig:diffAll1}\usebox{\diffAllone}}
\endminipage\hfill
\minipage[t]{0.48\textwidth}
  \center
  \includegraphics[scale=0.6]{diffAll2}
  \vspace{-5mm}
  \caption*{\label{fig:diffAll2}\usebox{\diffAlltwo}}
\endminipage\hfill
\end{figure}

\begin{figure}[H]
\minipage[t]{0.48\textwidth}
  \center
  \hspace*{20mm}\includegraphics[scale=0.6]{diffAll3}
  \vspace{-5mm}
  \caption*{\label{fig:diffAll2}\usebox{\diffAllthree}}
\endminipage\hfill
\end{figure}

\newpage\section*{Acknowledgements}
I would like to express my very great appreciation to Professor Stefan van Zwam, my research advisor, for teaching me about matroid theory from the ground up, and for his patient guidance and valuable suggestions during the planning and development of this research work. I would also like to thank Professor Rudi Pendavingh for his algorithm that allowed us to generate all signed-graphic representations of a given dyadic matroid.


\begin{thebibliography}{9}
\bibitem[Cam06]{C}
P. Camion, Unimodular module, Discrete Mathematics, 2006.
\bibitem[Mus10]{M}
A. Musitelli, Recognizing binet matrices, Ph.D Thesis, 2010.
\bibitem[Oxl11]{oxley}
J. G. Oxley, \emph{Matroid Theory}, Oxford University Press, New York, 2011.
\bibitem[Piv11]{pivotto}
I. Pivotto, Even cycle and even cut matroids, Ph.D Thesis, University of Waterloo, 2011.
\newcommand{\etalchar}[1]{$^{#1}$}
\bibitem[S{\etalchar{+}}09]{sage}
W.\thinspace{}A. Stein et~al., \emph{{S}age {M}athematics {S}oftware
({V}ersion 6.1.1)}, The Sage Development Team, 2014, {\tt
http://www.sagemath.org}.
\bibitem[Zwa09]{vZ}
S.H.M. van Zwam, Partial fields in matroid theory, Ph.D. thesis, Technische Universiteit Eindhoven, 2009. 
\bibitem[PZ07]{vZ2}
R.A. Pendavingh and S.H.M. van Zwam, $\sqrt[6]{1}$-signed graphs, 2007. 
\bibitem[PZ07a]{vZ3}
R.A. Pendavingh and S.H.M. van Zwam, Skew partial fields, multilinear representations of matroids, and a matrix tree theorem, 2007.
\end{thebibliography}
\end{document}